\documentclass[a4paper,10pt]{article}

\usepackage[charter]{mathdesign}

\usepackage{amsmath,amsthm}
\usepackage{dsfont}

\usepackage{xcolor}

\newtheorem{theorem}{Theorem}[section]

\newtheorem{lemma}[theorem]{Lemma}
\newtheorem{corollary}[theorem]{Corollary}
\newtheorem{remark}[theorem]{Remark}
\newtheorem{definition}[theorem]{Definition}

\newtheorem{assumption}[theorem]{Assumption}

\newcommand{\N}{\mathbb{N}}
\newcommand{\Z}{\mathbb{Z}}
\newcommand{\Q}{\mathbb{Q}}
\newcommand{\R}{\mathbb{R}}

\newcommand{\A}{\mathcal{A}}
\newcommand{\B}{\mathcal{B}}
\newcommand{\C}{\mathcal{C}}
\newcommand{\F}{\mathcal{F}}
\newcommand{\G}{\mathcal{G}}
\newcommand{\X}{\mathcal{X}}
\newcommand{\Y}{\mathcal{Y}}
\newcommand{\M}{\mathcal{M}}

\renewcommand{\P}{\mathbb{P}}
\newcommand{\E}{\mathbb{E}}

\newcommand{\law}[1]{\text{Law}(#1)}

\newcommand{\Pas}{\text{a.s.}}
\newcommand{\ind}{\mathds{1}}

\newcommand{\eps}{\varepsilon}
\newcommand{\la}{\lambda}
\newcommand{\ga}{\gamma}
\newcommand{\ka}{\kappa}
\newcommand{\dtv}{d_{\text{TV}}}

\newcommand{\dint}{\mathrm{d}}  % integrálás d-betű

\newcommand{\lfrf}[1]{\lfloor #1\rfloor}

\begin{document}
	
	\title{Markov chains in random environment with applications in queueing theory and machine learning
		\thanks{Both authors benefitted from the support of the ``Lend\"ulet'' grant LP 2015-6 of 
		the Hungarian Academy of Sciences.
		The second author was also supported by the NKFIH (National Research, Development and Innovation Office, Hungary) 
grant KH 126505.	
	}}

	\author{
		Attila Lovas\thanks{Alfr\'ed R\'enyi Institute of Mathematics and Budapest University of Technology and Economics, Budapest, Hungary}
		\and 
		Mikl\'os R\'asonyi\thanks{Alfr\'ed R\'enyi Institute of Mathematics, Budapest, Hungary}}
	
	\date{\today}
	
	\maketitle
	
	\begin{abstract}
			We prove the existence of limiting distributions for a large class of Markov chains 
		on a general state space in a random environment. We assume suitable versions of
		the standard drift and minorization conditions. In particular, the system dynamics
		should be contractive on the average with respect to the Lyapunov function and 
		large enough small sets should exist with large enough minorization constants. 
		We also establish that a law of large numbers
		holds for bounded functionals of the process. Applications to queuing systems, to machine learning algorithms 
		and to autoregressive processes are presented.
	\end{abstract}
	
	\section{Introduction}\label{sec:intro}

	Markov chains in stationary random environments (MCREs) with a general (not necessarily countable) state
	space appear in several branches of applied probability. Rough volatility models of mathematical finance (see \cite{cr,gjr}),
	queuing models with non-i.{i}.{d}.\ service times (see \cite{borovkov} and Section
	\ref{sec:queing} below) and sequential Monte Carlo methods (see Section \ref{sec:langevin} below) are
	prominent examples. It seems that existing studies on the ergodic theory of MCREs (such
	as \cite{kifer1,kifer2,sep,stenflo}) impose conditions that exclude the treatment of relevant models from the applications above. 
	
	The article \cite{rasonyi2018}, introducing new tools, managed 
	to establish the existence of limiting laws and ergodic theorems for certain 
	classes of MCREs which satisfy suitable versions of the standard drift 
	and minorization conditions of Markov chain theory (as presented e.g.\ in \cite{mt}).
	
	Assumption 2.2 of \cite{rasonyi2018}, however, severely restricted the scope of applications by requiring that
	the system dynamics is contractive whatever the state of the random environment is. The present study
	aims to remove this restriction: we require only that process dynamics is contractive
	\emph{on the average}, in the sense of Assumption \ref{as:LT} below.
	
	In Section \ref{sec:main} our main results are stated in an abstract framework. Two applications
	are worked out in detail in Sections \ref{sec:queing} and \ref{sec:langevin}. In 
	Section \ref{sec:queing}, we study a queuing model, where service times are not i.i.d. 
	In Section \ref{sec:langevin}, we treat the stochastic gradient Langevin 
	dynamics with stationary data, a sampling algorithm with important applications in
	machine learning, see \cite{wt,6}. Some ramifications are presented
	in Section \ref{sec:ram} and they are applied to linear systems in Section
	\ref{sec:lin}. Proofs are presented in Section \ref{sec:proofs}.

	\bigskip
	\noindent
	{\bf Notations and conventions.} Let $\R_{+}:=\{x\in\R:\, x\geq 0\}$
	and $\mathbb{N}^{+}:=\{n\in\mathbb{N}:\ n\geq 1\}$.
	Let $(\Omega,\F,\P)$ be a probability space. We denote by $\E[X]$
	the expectation of a random variable $X$. For $1\le p<\infty$, $L^p$ is used to denote the usual space of $p$-integrable real-valued random variables and $\Vert X \Vert_p$ stands for the $L^p$-norm of a random variable $X$.
	
	We fix a standard Borel space $(\X,\B)$. The set of
	probability Borel measures on $(\X,\B)$ are denoted by $\M_1$.
	The total variation metric on $\M_1$ is defined by 
	\begin{equation*}
	\dtv (\mu_1,\mu_2)=|\mu_1-\mu_2|(\X),\quad \mu_1,\mu_2\in\M_1,
	\end{equation*}
	where $|\mu_1-\mu_2|$ denotes the total variation of the signed measure $\mu_1-\mu_2$. 
	
	For $\mu_1,\mu_2\in\M_1$, let $\C (\mu_1,\mu_2)$ denotes the set of probability measures on $\B\otimes\B$ such that its respective marginals are $\mu_1$ and $\mu_2$. Then, $\dtv (\mu_1,\mu_2)$ can be expressed as twice the optimal transportation cost, between $\mu_1$ and $\mu_2$, that is
	\begin{equation}\label{eq:dtv}
	\frac{1}{2}\dtv (\mu_1,\mu_2) = \inf_{\ka\in\C (\mu_1,\mu_2)}
	\int_{\X\times\X} \ind_{x\ne y} \, \ka (\dint x, \dint y).
	\end{equation}

	In the sequel, we employ the convention that $\sum_{k}^{l}=0$ and $\prod_{k}^{l}=1$ whenever $k,l\in\Z$, $k>l$. 
	Lastly, $\langle \cdot\mid \cdot\rangle$ denotes the standard Euclidean inner product
	on finite dimensional vector spaces. For example, on $\R^d$, $\langle x\mid y\rangle = \sum_{i=1}^{d} x_i y_i$.
	
	\section{Main results}\label{sec:main}
	
	% Clarify the object of interest
	Let $(\Y,\A)$ be a measurable space and $Y:\Z\times\Omega\to\Y$ a strongly stationary $\Y$-valued 
	stochastic process which we interpret as the environment which influences the evolution of
	our main process of interest ($X$ below). 
	We consider a parametric family of stochastic kernels, that is a map $Q:\Y\times\X\times\B\to [0,1]$, where for all $B\in\B$ the function $(y,x)\mapsto Q(y,x,B)$ is $\A\otimes\B$-measurable and for all $(y,x)\in\Y\times\X$, $B\mapsto Q(y,x,B)$ is a 
	probability measure on $\B$.
	
	We assume that we are given the $\X$-valued process $X_t$, $t\in\N$ such that $X_0=x_0\in\X$
	is fixed and
	\begin{equation}\label{eq:Xdef}
	\P (X_{t+1}\in B\mid\F_t)=Q(Y_t,X_t,B)\quad\P-\Pas,\,t\in\N,
	\end{equation}
	where the filtration is
	\begin{equation*}
	\F_t = \sigma (X_s,\, 0\le s\le t;\, Y_s,\, s\in\mathbb{Z}),\, t\in\N .
	\end{equation*}
	Let $\mu_t\in\M_1$ denote the law of $X_t$ for $t\in\N$.
	
	We aim to study the ergodic properties of $X_t$ and the convergence of $\mu_t$ to a limiting law as $t\to\infty$ under various assumptions.
	
	% Assumptions
	\begin{definition}\label{def:act}
	Let $P:\X\times\B\to [0,1]$ be a probabilistic kernel. For
	a bounded measurable function $\phi:\X\to\R$, we define
	\begin{equation*}
	[P\phi](x)=\int_\X \phi(z) P(x,\dint z),\,x\in\X.
	\end{equation*}
	This definition makes sense for any non-negative measurable $\phi$, too.
	\end{definition}
	Consistently with Definition \ref{def:act}, for $y\in\Y$, $Q(y)\phi$ will refer to the action of the kernel $Q(y,\cdot,\cdot)$ on $\phi$.
	
	\begin{assumption}\label{as:drift}(Drift condition)
		Let $V:\X\to\R_{+}$ be a measurable function.
	We assume that there are measurable functions $K,\ga:\Y\to (0,\infty)$ such that, for all $x\in\X$ and $y\in\Y$,
	\begin{equation*}
	[Q(y)V](x)\le \ga (y)V(x)+K(y).
	\end{equation*}
	Furthermore, we may and will assume that $K(\cdot)\ge 1$.
	\end{assumption}

In contrast with the drift condition used in \cite{rasonyi2018} 
(cf. Assumption 2.2 on page 2), the domain of $\ga$ is $\Y$ and not $\N$.{}
Moreover, it is possible that $\ga (y)\geq 1$ holds for certain $y\in\Y$. This relaxation allows the inclusion
of several models that were intractable using the results of \cite{rasonyi2018}. 
Although $\ga(y)\geq 1$ may hold, in the next assumption we require that the system
dynamics, on average, is contracting in the long run.
	
	\begin{assumption}\label{as:LT}(Long-time contractivity condition)
		We assume that	
		\begin{equation*}
		\bar{\ga}:=\limsup_{n\to\infty}\E^{1/n}\left(K(Y_0)\prod_{k=1}^{n}\ga (Y_k)\right) < 1.
		\end{equation*}
	\end{assumption}

	The next assumption stipulates the existence of suitable ``small sets''. It corresponds to
	Assumption 2.5 in \cite{rasonyi2018} but we need a different formulation here.

	\begin{assumption}\label{as:minor}(Minorization condition)
	Let $\la (\cdot)$, $K(\cdot)$ be as in Assumption \ref{as:drift}.
	We assume that for some $0<\eps<1/\bar{\ga}^{1/2}-1$, there is a measurable function $\alpha:\Y\to [0,1)$ and a probability kernel $\ka: \Y\times\B\to [0,1]$ such that, for all $y\in\Y$ and $A\in\B$,
	\begin{equation}\label{maudit}
		\inf_{x\in V^{-1}([0,R(y)])} Q(y,x,A)\ge (1-\alpha (y)) 
		\ka (y,A), \text{ where } R(y)=\frac{2K(y)}{\eps\ga (y)}	
	\end{equation}
	and $V^{-1}([0,R(y)])\neq\emptyset$.
	\end{assumption}
	
	\begin{remark}{\rm If there is $x\in\mathcal{X}$ with $V(x)=0$ then
	$V^{-1}([0,R(y)])\neq\emptyset$ automatically holds.} 
	\end{remark}

	The larger $\alpha(y)$ is, the weaker condition \eqref{maudit} is. Hence one needs to
	control the probability of $\alpha(Y_{0})$ approaching $1$ and enforce the 
	``smallness'' of $\alpha(Y_{0})$. 
	This is the content of the following condition
	which will play a very important role in our convergence estimates.
	\begin{assumption}(Smallness condition) \label{as:myas}
		There exists $0<\theta<1$ such that
		\begin{equation*}
			\lim_{n\to\infty} \E^{1/n^\theta}\left(\alpha (Y_0)^n\right) = 0.
		\end{equation*}
	\end{assumption}

	\begin{remark}{\rm Assumption \ref{as:myas} clearly holds if $\alpha(\cdot)\equiv \alpha\in (0,1)$
	is a constant. To see a simple non-constant example, let $\mathcal{Y}:=\N_{+}$
	and assume that 
		\begin{equation}
		\inf_{x\in V^{-1}([0,R(y)])} Q(k,x,A)\ge \frac{1}{k^{\chi}} 
		\ka (k,A)	
	\end{equation}
	holds for a suitable kernel $\ka$, for all $k\in\N^{+}$ with some $\chi\geq 1$.{}
	Note that $1-e^{-1/k^{\chi}}\leq 1/k^{\chi}$ hence inequality \eqref{maudit} in the minorization
	condition holds with $\alpha(k)=e^{-1/k^{\chi}}$. Let $Y_{0}$ satisfy
	$P(Y_{0}=k)=e^{-\delta k}c_{\delta}$, $k\in\N_{+}$
	with a suitable normalizing constant $c_{\delta}>0$.
	It follows that
	\begin{eqnarray*}
	E[\alpha(Y_{0})^{n}]&\leq& c_{\delta}\sum_{k=1}^{\lfloor {n}^{\frac{1}{2\chi}}\rfloor} e^{-n/k^{\chi}}
	e^{-{\delta}k} + c_{\delta}\sum_{k=\lfloor {n^{\frac{1}{2\chi}}}\rfloor+1}^{\infty} e^{-n/k^{\chi}}
	e^{-{\delta}k}\\
	&\leq& c_{\delta}\lfloor {n}^{\frac{1}{2\chi}}\rfloor e^{-{n}^{\frac{1}{2}}}
+ C_{\delta} e^{-\delta (\lfloor {n}^{\frac{1}{2\chi}}\rfloor+1) }
	\end{eqnarray*}
for a suitable constant $C_{\delta}>0$. Then Assumption \ref{as:myas} clearly holds
with any $0<\theta<\frac{1}{2\chi}$.	
}
		
	\end{remark}
	
	% Presenting the results
	Now come the main results of the present paper: with the above presented assumptions, 
	the law of $X_t$ converges to a limiting law as $t\to\infty$, moreover, bounded functionals of
	$X_t$ admit ergodic behavior provided that $Y_t$ is ergodic.
	\begin{theorem}\label{thm:TV}
		 Under Assumptions \ref{as:drift}, \ref{as:LT}, \ref{as:minor} and \ref{as:myas}, there exists a universal probability law $\mu_{\ast}$, independent of $x_0$, such that $\mu_{N}\to\mu_{\ast}$ in
		 total variation as $N\to\infty$. 
		 
		 More precisely, for any $1/2<\la<1$, there exist $c(\la),\nu (\la)>0$ such that
		 \begin{align}
		 	& \dtv (\mu_N,\mu_{\ast}) \nonumber \le
		 	\\ & 2\sum_{n=N}^{\infty}
		 	\left[
		 	\E\left(\max_{0\le k<\lfrf{n^{1/3}}^3} \alpha (Y_k)^{\lfrf{n^{1/3}}-1}\right)
		 	+\bar{\ga}^{\left(\la-\frac{1}{2}\right)\lfrf{n^{1/3}}^2-n^{1/3}}
		 	+cn^{2/3} e^{-\nu n^{1/3}}
		 	\right] \label{zoroaster}
		 \end{align}
	holds for all $N\in\N$.
	\end{theorem}

	\begin{remark}\label{simplify}{\rm To help decipher the expression \eqref{zoroaster}, we remark that setting $\la=3/4$,{}
	\eqref{zoroaster} is easily seen to be dominated by
$$
2\sum_{n=N}^{\infty}
		 	\left[
		 	\E\left([\max_{0\le k\leq n} \alpha (Y_k)]^{n^{1/3}-2}\right)
		 	+\bar{\ga}^{n^{2/3}/8}
		 	+cn^{2/3} e^{-\nu n^{1/3}}
		 	\right]
$$		
for $N\geq 216$, with suitable constants $c,\nu>0$. In the particular case where $\alpha(\cdot)$ is constant,
the latter expression is $O(e^{-\zeta n^{1/3}})$ for some $\zeta>0$, as easily seen.}
	\end{remark}

To make the explanation self-contained, we give the definition of ergodic environment which we will use in the sequel. Let $\Y^\Z$ be the space of trajectories, $\A^{\otimes\Z}$ the $\sigma$-algebra generated by cylinder sets, and $\law{Y}$ the distribution on $\Y^\Z$ induced by $Y$. We introduce the usual \emph{left shift operation} $S:\Y^\Z\to\Y^\Z$ acting on trajectories like $(S\mathbf{y})_j=\mathbf{y}_{j+1}$, $j\in\Z$. 
As the environment is assumed to be strongly stationary, the probability $\law{Y}$ is invariant under the transformation $S$ and thus $(\Y^\Z,\A^{\otimes\Z},\law{Y},S)$ forms a dynamical system.
We say that an event $A\in\A^{\otimes\Z}$ is invariant iff $S^{-1}(A)=A$ and the process $Y$
is \emph{ergodic} if the $\sigma$-algebra generated by invariant events $\mathcal{I}$, which is a sub-$\sigma$-algebra of $\A^{\otimes\Z}$, is trivial for $\law{Y}$ i.e. for all $A\in\mathcal{I}$,
$\law{Y}(A)\in\{0,1\}$.
	\begin{theorem}\label{thm:LLN}
	 Let Assumptions \ref{as:drift}, \ref{as:LT}, \ref{as:minor} and \ref{as:myas} be in force. If $Y$ is \emph{ergodic}, 
	 then for any bounded and measurable $\Phi:\X\to\R$ 
	 \begin{equation}\label{torta}
	 	\frac{\Phi (X_1)+\ldots + \Phi (X_N)}{N} \to \int_\X \Phi (z)\,\mu_{\ast}(\dint z), \,N\to\infty
	 \end{equation}
	 holds in $L^p$, for all $1\le p<\infty$, where $\mu_{*}$ is as in Theorem \ref{thm:TV} above.
	\end{theorem}
	
	\begin{remark}{\rm Since $\Phi$ is bounded, convergence in \eqref{torta} takes place in probability
iff it happens in $L^{p}$ for all $1\leq p<\infty$. We preferred the current formulation of
Theorem \ref{thm:LLN} since we obtain $L^{p}$ rates during the proofs, see also Remark \ref{borkoff}.
However, the expressions for these rates are too complicated to be stated here.}
\end{remark}

\begin{remark}{\rm The above theorems should be compared to Theorems 2.13 and 2.15 of \cite{rasonyi2018}.
Our results are definitely stronger in the sense that \cite{rasonyi2018} requires $\gamma(\cdot)<1$ while
we do not. However, those results are not \emph{subsumed} by the present ones, they are merely \emph{complemented}.
For instance, Assumption 2.5 of \cite{rasonyi2018} applies to certain cases where our Assumption 2.4 does not.
It is also unclear whether the assumptions of \cite{rasonyi2018} imply our Assumption 2.3 in general.}
\end{remark}

%%%%%%%%%%%%%%%%%%%%%%%%%%%%%%%%%%%%%%%%%%%%%%%%%%%%%%%%%%%%%%%%%
%%%%%%%%%%%%%%%%%%%%%%%%%%%%%%%%%%%%%%%%%%%%%%%%%%%%%%%%%%%%%%%%%
%%%%%%%%%%%%%%%%%%%%%%%%%%%%%%%%%%%%%%%%%%%%%%%%%%%%%%%%%%%%%%%%%

\section{Ramifications}\label{sec:ram}

Although many stochastic models, including all discretized diffusion processes (as well as 
discretely sampled diffusions), fall within the scope of our framework in the previous section, 
certain classes of processes are not covered in this way.
For instace, vector autoregressive models fail to satisfy the one-step drift and minorization conditions.
 
To make our techniques applicable to a larger class of models, we generalize Theorems \ref{thm:TV} and \ref{thm:LLN}
to cases where we require the drift and minorization conditions to hold only after several steps.
In Section \ref{sec:lin} we will apply these theorems to vector autoregressive models.

To keep complexity at a tolerable level, we restrict ourselves to the case of constant $K$ in the drift 
condition below.
We also need to assume \eqref{eq:extra}, inspired by Assumption 6 in \cite{magic5}.

\begin{assumption}\label{as:drift1}
	Let $V:\mathcal{X}\to\R_{+}$ be a measurable function.
	We assume that there is an integer $p\geq 1$, a measurable function $\ga:\mathcal{Y}^p\to (0,\infty)$ and a constant $K\geq 1$
	such that, for all $x\in\mathcal{X}$ and $(y_1,\ldots,y_p)\in\mathcal{Y}^p$,
	\begin{equation*}
	[Q(y_p)Q(y_{p-1})\ldots Q(y_1)V](x)\le \ga (y_1,\ldots,y_p)V(x)+K
	\end{equation*}
	and
	\begin{equation}\label{eq:extra}
	[Q(y)V](x)\le K V(x)+K.
	\end{equation}
\end{assumption}

\noindent
The corresponding long-time contractivity condition reads as follows.
\begin{assumption}\label{as:LT1}
	We assume that	
	\begin{equation*}
	\bar{\ga}:=\limsup_{n\to\infty}E^{1/n}\left(\prod_{i=1}^{n}\ga (Y_{(i-1)p+1},\ldots,Y_{ip})\right)< 1.
	\end{equation*}
\end{assumption}

Similarly, in the minorization condition below, 
the lower bound is also assumed independent of the environment for simplicity, thus there is no need for a 
generalized version of the smallness condition i.e.\ Assumption \ref{as:myas}.

\begin{assumption}\label{as:minor1}
	Let $\ga (\cdot)$, $K(\cdot)$ be as in Assumption \ref{as:drift1}.
	We assume that for some $0<\epsilon<1/\bar{\ga}^{1/2}-1$, there is a constant $0\le\alpha <1$ and a probability  
	$\ka$ such that, for all $(y_1,\ldots,y_p)\in\mathcal{Y}^p$ and $A\in\mathcal{B}(\mathcal{X})$,
	\begin{equation}\label{maudit2}
	\inf_{x\in V^{-1}([0,R(y_1,\ldots,y_p)])} [Q(y_p)Q(y_{p-1})\ldots Q(y_1)](x,A)\ge (1-\alpha) 
	\ka (A), 	
	\end{equation}
	where $R(y_1,\ldots,y_p)=\frac{2K}{\epsilon\ga (y_1,\ldots,y_p)}$ and we require $V^{-1}([0,R(y_1,\ldots,y_p)])\neq\emptyset$.
\end{assumption}

We can now state the following complement to Theorems \ref{thm:TV} and \ref{thm:LLN}. 

\begin{theorem}\label{thm:gen}
	Under Assumption \ref{as:drift1}, \ref{as:LT1} and \ref{as:minor1}, 
	$\mu_{N}$ converges to a limiting law, independent of $x_0$, at speed $O(e^{-cN^{1/3}})$. Furthermore, 
	if $Y$ is ergodic, then the law of large numbers holds as stated in
	Theorem \ref{thm:LLN}.		
\end{theorem}

	%%%%%%%%%%%%%%%%%%%%%%%%%%%%%%%%%%%%%%%%%%%%%%%%%%%%%%%%%%%%%%%%
	%%%%%%%%%%%%%%%%%%%%%%%%%%%%%%%%%%%%%%%%%%%%%%%%%%%%%%%%%%%%%%%%
	%%%%%%%%%%%%%%%%%%%%%%%%%%%%%%%%%%%%%%%%%%%%%%%%%%%%%%%%%%%%%%%%
	
	\section{A queuing model}\label{sec:queing}
	% Application 1

	We consider a single-server queuing model where customers are numbered by $n\in\N$. The time between the arrival of 
	customers $n+1$ and $n$ is described by the random variable $\eps_{n+1}$,
	for each $n\in\N$. 
	The service time for customer $n$ is given by the random variable $Y_n$, for $n\in\N$.
	
	The waiting time $W_n$ of customer $n$ satisfies the Lindley recursion
	\begin{equation}\label{eq:Lindley}
	W_{n+1}=(W_n+Y_{n}-\eps_{n+1})_+,\ n\in\N,
	\end{equation}
	with $W_0:=0$ (we start with an empty queue hence the $0$th customer does not
	need to wait at all). The textbook example is 
	when $(Y_n)_{n\in\N}$ and $(\eps_n)_{n\in\N^+}$ are i.i.d. sequences independent of each other.
	In that case $W_n$ is a Markov chain with state space $\mathbb{R}_{+}$
	whose ergodic properties have been extensively studied. 
	Here we are interested in a more general
	setting where the process $(Y_n)_{n\in\N}$ is assumed merely stationary.
	
	\smallskip
	The following condition is standard: in a stable system service times should be shorter on the average than inter-arrival times.
	\begin{assumption}\label{as:eps1geY0}
	The sequence of $\R_{+}$-valued inter-arrival times $\eps_n$,
	$n\in\mathbb{Z}$ is i.i.d.\ and we have	
	\begin{equation*}
		\E[Y_0]<\E[\eps_1].
	\end{equation*}	
	\end{assumption}

	\begin{assumption}\label{as:YLargeDev} For some $M>0$, the sequence of service times is included
		in a strict sense $[0,M]$-valued stationary process 
		$Y_n$, $n\in\Z$ which is independent of $(\eps_n)_{n\in\mathbb{Z}}$. There is $\eta>0$ such that the limit 
		\begin{equation}\label{labbbe}
		\Gamma(\alpha):=\lim_{n\to\infty}\frac{1}{n}\ln \E e^{\alpha(Y_1+\ldots+Y_n)}
		\end{equation}
		exists for all $\alpha\in (-\eta,\eta)$ and $\Gamma$ is differentiable on $(-\eta,\eta)$. 
		%Furthermore,
		%\begin{equation}\label{morba}
		% \E[e^{\delta Y_1}]<\infty
		%\end{equation}
		%holds for some $\delta>0$.
	\end{assumption}
	
\begin{remark}
{\rm The assumption above is clearly inspired by the G\"artner-Ellis theorem. It follows that sufficient conditions for its fulfillment can be deduced from the literature about large deviation principles.
	For instance, if $Y_n=\phi(Z_n)$ for some bounded measurable 
	$\phi:\R^m\to\R_{+}$ and an $\R^m$-valued geometrically ergodic 
	Markov chain $Z_n$, $n\in\Z$ started from its invariant distribution then 
	\eqref{labbbe} holds true for some $\eta>0$, see Theorem 4.1 of \cite{KM1} for a precise formulation. 
	Thus Theorem \ref{thm:queue} below is applicable to a large class of models.
	
	Infinite moving average processes serve as an example of non-Markovian service time processes.
	For instance, let 
	$Y_t=\sum_{i=-\infty}^{\infty} a_i\zeta_{t-i}$, 
			where $\zeta_i$, $i\in\Z$ are independent and identically distributed 
			$\R_{+}$-valued bounded random variables, $a_i\ge 0$, $i\in\Z$ and
			$\sum_{i=-\infty}^{\infty} a_i<\infty$. 
			Assumption \ref{as:YLargeDev} is satisfied for this process by Theorem 2.1 of \cite{dehling}.
			
			Under suitable conditions, it is possible to relax the boundedness condition on the process $Y$ in Assumption 
			\ref{as:YLargeDev}. Due to tedious technicalities, this is not pursued here.}
\end{remark}

	Now, we turn to the verification of Assumption \ref{as:drift} and \ref{as:LT} under
	the previous two conditions.
	
	\begin{lemma}\label{lem:queuing:DriftAndLT}
	Let	Assumptions \ref{as:eps1geY0} and \ref{as:YLargeDev} be in force. Then there exists
	$\bar{\alpha}>0$ such that for
	\begin{align*}
	V(w)    &:= e^{\bar{\alpha}w}-1,\, w\ge 0,\\	
	\ga (y) &:=\E\left[e^{\bar{\alpha}(y-\eps_1)}\right],\,y\ge 0,\\
	K       &:=e^{\bar{\alpha}M},
	\end{align*} 
	\begin{equation}\label{driff}
	[Q(y)V](w)\le \ga (y)V(w)+K
	\end{equation}
	holds for all $y\in\mathcal{Y}:= [0,M]$, $w\in\mathcal{X}:=\R_{+}$,
	where $Q$ is defined as
	\begin{equation*}
	Q(y,w,A):=\P \left[\left(w+y-\eps_1\right)_+\in A\right],\,
	y\in [0,M],w\in\R_{+},\,A\in\B\left(\R_{+}\right).
	\end{equation*}
	Furthermore,
	\begin{equation*}
	\bar{\ga}:=\limsup_{n\to\infty}\E^{1/n}\left(K\prod_{k=1}^{n}\ga (Y_k)\right) < 1.
	\end{equation*}
	\end{lemma}
	\begin{proof}
			Define $\la(\alpha):=\Gamma(\alpha)+\ln(\E[e^{-\alpha \eps_1}])$.
			The functions 
			\begin{equation*}
			\la_n(\alpha):=\frac{1}{n}\ln \E\left[e^{\alpha\sum_{j=1}^n(Y_{j-1}-\eps_j)}\right],\ \alpha\in (-\eta,\eta),\ n\in\N^+
			\end{equation*}
			are finite and differentiable. They are also clearly convex. Define 
			\begin{equation*}
			\psi_n(\alpha):= \E\left[\frac{e^{\alpha\sum_{j=1}^n(Y_{j-1}-\eps_j)}-1}{\alpha}\right],\ \alpha\in (0,\eta),\ n\in\N^+.
			\end{equation*}
			By the Lagrange mean value theorem and measurable selection, there exists a random variable $\xi_n(\alpha)\in [0,\alpha]$ such that
			\begin{equation*}
			\psi_n(\alpha)= \E\left[\left(\sum_{j=1}^n(Y_{j-1}-\eps_j)\right)e^{\xi_n(\alpha)\sum_{j=1}^n(Y_{j-1}-\eps_j)}\right].
			\end{equation*}
			Here 
			\begin{equation*}
			\left(\sum_{j=1}^n (Y_{j-1}-\eps_j)\right)e^{\xi_n(\alpha)\sum_{j=1}^n(Y_{j-1}-\eps_j)}\leq \left(
			\sum_{j=1}^n Y_{j-1}\right) e^{\eta \sum_{j=1}^n Y_{j-1}},
			\end{equation*}
			which is uniformly bounded in $\alpha\in (0,\eta)$ (for $n$ fixed). Hence reverse Fatou's lemma shows that
			\begin{equation*}
			\limsup_{\alpha\to 0+}\psi_n(\alpha)\leq \E\left[\sum_{j=1}^n (Y_{j-1}-\eps_j)\right]=n \E\left[Y_0-\eps_1\right].
			\end{equation*}
			This implies that, for all $n\geq 1$, 
			$\la_n'(0)=\frac{1}{n}\lim_{\alpha\to 0+}\psi_n(\alpha)\leq{\E\left[Y_0-\eps_1\right]}$.
			
			Since $\la_n(\alpha)\to\la(\alpha)$ for $\alpha\in (-\eta,\eta)$ by Assumption \ref{as:YLargeDev} it follows from Theorem 25.7
			of \cite{rockafellar} that also $\la_n'(0)\to\la'(0)$ hence $\la'(0)<0$ by Assumption \ref{as:eps1geY0}. By Corollary 25.5.1 of \cite{rockafellar}, differentiability of $\la$ implies its
			\emph{continuous} differentiability, too. Hence
			from $\la(0)=0$ and $\la'(0)<0$ we obtain that there exists $\bar\alpha>0$ satisfying 
			\begin{equation}\label{matra}
			\lim_{n\to\infty}\frac{1}{n}\ln \E e^{\bar{\alpha}(Y_0+\ldots+Y_{n-1})-\bar{\alpha}(\eps_1+\ldots+\eps_n)}<0.
			\end{equation}
			
			Now using the Lyapunov function $V(w)=e^{\bar{\alpha}w}-1$, $w\geq 0$
			and $\ga(y)=\E\left[e^{\bar{\alpha}(y-\eps_1)}\right],\ y\geq 0$
			we arrive at
			\begin{align*}
				[Q(y)V](w) &= \E[V([w+y-\eps_1]_+)]\\
				&=
				\E[e^{\bar{\alpha}(w+y-\eps_1)_+}]\le
				\E[e^{\bar{\alpha}(w+y-\eps_1)}]
				\\
				&\le
				\ga(y)(e^{\bar{\alpha}w}-1) + \ga (y)\leq \ga(y)V(w)+ e^{\bar{\alpha}M},
			\end{align*}
			so \eqref{driff} holds with $K$ as defined above.
			By \eqref{matra}, the long-time contractvity condition
			also holds:
			\begin{equation}\label{juj}
			\limsup_{n\to\infty} \E^{1/n}[K\ga(Y_1)\ldots\ga(Y_n)]<1,
			\end{equation}
			which completes the proof.		
	\end{proof}
	
	Now we present another assumption on the inter-arrival times which will be needed to
	show the minorization condition. 
	\begin{assumption}\label{as:queue:minor}
		One has $\P \left(\eps_1\ge\tau\right)>0$ for $$
		\tau:=M+\frac{4}{\frac{1}{\bar{\ga}^{1/2}}-1}.
		$$
	\end{assumption}	
	Notice that, for unbounded $\eps_1$, Assumption \ref{as:queue:minor} automatically
	holds.
	Now let us turn to the verification of the minorization condition under the assumption above.
	\begin{lemma}\label{lem:queuing:minor}
		Let Assumptions \ref{as:eps1geY0}, \ref{as:YLargeDev} and \ref{as:queue:minor} be in force. Choose
		$\epsilon:=(1/\bar{\ga}^{1/2}-1)/2$. Then there is $\alpha\in (0,1)$ such that,
		for all $y\in [0,M]$ and $A\in\B\left(\R_{+}\right)$,
		\begin{equation*}
			\inf_{w\in V^{-1}([0,R(y)])} Q(y,w,A)\ge (1-\alpha) \delta_0 (A),\,
			\text{where }R(y)=\frac{2 K(y)}{\epsilon\ga (y)}
		\end{equation*}
		and $\delta_0$ is the one-point mass concentrated on $0$.
	\end{lemma}
\begin{proof} Note that $R(y)\equiv R:=\frac{2}{\epsilon}$.
		\begin{align*}
			Q(y,w,A) &= \P \left(\left[w+y-\eps_1\right]_+\in A \right) \\
			         &\ge \P \left(\left[w+y-\eps_1\right]_+=0 \right)\delta_0(A) \\
			         &= \left(1-\P \left(w+y-\eps_1 > 0 \right)\right)\delta_0(A) \\
			         &\ge \left(1-\P \left(R+M-\eps_1 > 0 \right)\right)\delta_0(A)
		\end{align*}
so we may set $\alpha:=\P \left(\eps_1<\tau\right)\}<1$, see Assumption \ref{as:queue:minor}.
\end{proof}

%	\begin{lemma}\label{lem:queuing:minor}
%		Let Assumption \ref{as:queue:minor} be in force. Then there exists
%		$0<\eps<1/\bar{\ga}^{1/2}-1$ and $\alpha\in (0,1)$ such that,
%		for all $y\in [0,\infty)$ and $A\in\B\left((0,\infty]\right)$,
%		\begin{equation*}
%			\inf_{w\in V^{-1}([0,R(y)])} Q(y,w,A)\ge (1-\alpha) \delta_0 (A),\,
%			\text{where }R(y)=\frac{2}{\eps\ga (y)}
%		\end{equation*}
%		and $\delta_0$ is the one-point mass concentrated on $0$.
%	\end{lemma}
%	\begin{proof}
%		By Assumption \ref{as:queue:minor}, there exists $0<\eps<1/\bar{\ga}^{1/2}-1$
%		such that, for 
%		\begin{equation*}
%		\tau:=-\frac{1}{\bar{\alpha}}\log\left(\frac{\eps}{2}\E\left[e^{-\bar{\alpha}\eps_1}\right]\right),\quad\alpha:=\max\left(1/2,\P (\eps_1<\tau)\right)<1.
%		\end{equation*}
%		Let $y\in [0,\infty)$ be arbitrary and fixed and define $R(y)=\frac{2}{\eps\ga (y)}$. Then for $w\in V^{-1}([0,R(y)])$ and $A\in\B\left((0,\infty]\right)$, we can write
%		\begin{align*}
%			Q(y,w,A) &= \P \left(\left[w+y-\eps_1\right]_+\in A \right) \\
%			         &\ge \P \left(\left[w+y-\eps_1\right]_+=0 \right)\delta_0(A) \\
%			         &= \left(1-\P \left(w+y-\eps_1 > 0 \right)\right)\delta_0(A) \\
%			         &\ge (1-\alpha)\delta_0(A),
%		\end{align*}
%		where we used  
%		$w\in V^{-1}([0,R(y)])\Leftrightarrow e^{\bar{\alpha}(w+y)}\le\frac{2}{\eps}\E(e^{-\bar{\alpha}\eps_1})^{-1}$ and
%		thus
%		\begin{equation*}
%		\P \left(w+y-\eps_1 > 0 \right)\le\P (\eps_1<\tau)\le\alpha
%		\end{equation*}
%		which completes the proof.
%	\end{proof}
		
	Theorem \ref{thm:TV} allows us to deduce that the queuing system 
	in consideration converges to a stationary state and an
	ergodic theorem is valid. Theorem \ref{thm:queue} below opens the door for the statistical analysis 
	of such systems.
	\begin{theorem}\label{thm:queue} 
		Under Assumptions \ref{as:eps1geY0}, \ref{as:YLargeDev} and \ref{as:queue:minor} there exists a probability $\mu_{*}$ on $\mathcal{B}(\R_{+})$, independent of the initial length of the queue, such that 
		\begin{equation*}
		\dtv(\law{W_n},\mu_{*})\leq c_1e^{-c_2 n^{1/3}},
		\end{equation*}
		for some $c_1,c_2>0$. Furthermore, if $\left(Y_n\right)_{n\in\Z}$
		is ergodic, then for an arbitrary measurable and bounded $\Phi:\R_{+}\to\R$,
		\begin{equation}\label{torta2}
		\frac{\Phi(W_0)+\ldots+\Phi(W_{n-1})}{n}\to \int_{\R_{+}} \Phi(z)\mu_{*}(\dint z),
		\end{equation}
		in $L^p$, for all $1\le p<\infty$.
	\end{theorem}
	\begin{proof}
		According to Lemma \ref{lem:queuing:DriftAndLT} and \ref{lem:queuing:minor},
		Assumptions \ref{as:drift}, \ref{as:LT}, \ref{as:minor} and \ref{as:myas}
		are satisfied hence by Theorem \ref{thm:TV} and Remark \ref{simplify} the first statement follows.
		By Theorem \ref{thm:LLN}, the second statement also holds true.
	\end{proof}

		%there exists a probability $\mu_{*}$ on $\mathbb{R}_{+}$ such that,
		%for any $1/2<\la<1$ there are $c(\la),\nu (\la)>0$ satisfying
		%\begin{align*}
	%		\dtv(\law{W_n},\mu_{\ast})\le
%			2\sum_{k=n}^{\infty}
%			\left[
%			\alpha^{\lfrf{k^{1/3}}-1}
%			+\bar{\ga}^{\left(\la-\frac{1}{2}\right)\lfrf{k^{1/3}}^2-k^{1/3}}
%			+ck^{2/3} e^{-\nu k^{1/3}}
%			\right],
%		\end{align*}
%		where $\alpha$ is as in Lemma \ref{lem:queuing:minor}.
		
%		Clearly, the $k$th term in the above sum is of the order
%		$O(e^{-c_0 k^{1/3-\varrho}})$ for arbitrarily small $\varrho>0$ and for some $c_0=c_0(\varrho)>0$,
%		hence we obtain the claimed convergence rate.

	\begin{remark}{\rm It is known that 
		$\law{W_n}$ converges to a limiting
		distribution under rather mild conditions, see Example 14.1 on page 189 of \cite{borovkov}. 
		Details of this approach seem to be available only in Russian, see \cite{bborovkov}. 
		To be more explicit, applying Theorem 4 on page 25 of \cite{bborovkov72} (also in Russian) to our settings gives the following upper bound
		\begin{equation}\label{eq:borovkov}
			\dtv(\law{W_n},\mu_{\ast})\le \P \left(\min_{0<k<n} S_k>\max (W_1,W_0'+\xi_0)\right),
		\end{equation}
		where $(S_n)_{n\in\N}$ is defined as
		\begin{align*}
		S_0 &= 0\\
		S_n &= \sum_{k=1}^{n} (Y_{k}-\eps_k),\quad n\ge 1
		\end{align*} 
		moreover $W_{0}'=\sup_{k\in\mathbb{N}^{+}}(Y_{-k}-\eps_{-k})$. 
%		$W_n'$, $n\in\N$ is a strongly stationary process with which the sequence of waiting times is
%		coupled i.e. exists an almost surely finite integer valued random variable such that
%		$W_n=W_n'$ for $n\ge\tau$. 
%		A short and elegant proof for the existence of the process $W_n'$, $n\in\N$ and 
%		the coupling effect can be found in \cite{gyorfi2002}. 
However, the main drawback
		of this upper bound is that the expression standing on the right hand side 
		of \eqref{eq:borovkov} is still too general and non-informative. Therefore, we think that, 
		Theorem \ref{thm:queue} above is the first result providing a tractable rate of convergence in this setting.
}
	\end{remark}

	%%%%%%%%%%%%%%%%%%%%%%%%%%%%%%%%%%%%%%%%%%%%%%%%%%%%
	%%%%%%%%%%%%%%%%%%%%%%%%%%%%%%%%%%%%%%%%%%%%%%%%%%%%
	%%%%%%%%%%%%%%%%%%%%%%%%%%%%%%%%%%%%%%%%%%%%%%%%%%%%
	
	\section{Stochastic gradient Langevin algorithm}\label{sec:langevin}
	% Application 2
	
	We consider, for some $\la>0$,
	$$
	\theta_{n+1}=\theta_n-\la H(\theta_n,Y_n)+\sqrt{\la}\xi_{n+1},
	$$
	where $\xi_n$, $n\geq 1$ is an independent sequence of standard $d$-dimensional
	Gaussian random variables, $Y_n$, $n\in\Z$ is a $\R^m$-valued
	strict sense stationary process and $H:\R^d\times\R^m\to\R^d$
	a measurable function. We assume that $(Y_n)_{n\in\Z}$ and $(\xi_n)_{n\in\N^+}$ are independent, $\theta_{0}\in\mathbb{R}^{d}$
	is a constant. 
	
	This algorithm is called ``stochastic gradient Langevin dynamics'' (SGLD). Suggested by \cite{wt}, it has recently become widely used for
	sampling from high-dimensional probability distributions. More precisely,
	let $U:\R^d\to\R_+$ be differentiable with derivative $h=\nabla U$
	such that $h(\theta)=E[H(\theta,Y_0)]$. For $\la$ small and $n$ large,
	$\mathrm{Law}(\theta_n)$ is expected to be close to the probability defined by
	$$
	\pi(A)=\frac{\int_A e^{-U(\theta)}d\theta}{\int_{\R^d} e^{-U(\theta')}d\theta'},\ A\in\mathcal{B}(\R^d),
	$$
see e.g.\ \cite{wt,6}.	
The literature on SGLD is abundant but practically all studies assume that $Y_n$, $n\in\Z$
	are i.i.d. For the case where the step size $\la_{n}$ is decreasing, it has been 
	shown in \cite{ttv} that, under suitable
	assumptions, the averages
	$$
	D_n:=\frac{\Phi(\theta_0)+\ldots+\Phi(\theta_{n-1})}{n}
	$$
	converge almost surely to $D:=\int_{\R^d}\Phi(z)\pi(dz)$. 
	In the case of fixed $\la$, \cite{vzt} estimated the $L^2$ distance between $D_n$ and $D$.
	
	In the present article we keep $\la$ fixed and establish a novel result:
	the SGLD recursion converges to a limiting law $\mu(\la)$ (in total variation) and $D_n$ tends to $\int_{\R^d}\Phi(z)\mu(\la)(dz)$ in $L^p$, $1\le p<\infty$. As far as we know this ergodic property has not
	yet been pointed out, even in the case of i.i.d.\ $Y_n$, $n\in\Z$. We can now prove it for
	a broad class of stationary processes $Y_n$, $n\in\Z$. We think of $Y_n$ as an observed data
	sequence. As these are rarely i.i.d.\ in practice, Theorem \ref{guar} below
	formulates strong theoretical support for the use of SGLD with possibly dependent data.

	The following standard dissipativity condition is required, see e.g.\ \cite{raginsky}.
	
	\begin{assumption}\label{immeasurable} There is a measurable 
		$\Delta:\R^m\to\R$ and $b\geq 0$
		such that, for all $\theta\in\R^d$ and $y\in\R^m$,
		$$
		\langle H(\theta,y),\theta\rangle\geq \Delta(y)|\theta|^2-b.
		$$
		Furthermore, $E[\Delta(Y_0)]>0$. We may and will assume that $\Delta$ is a bounded function.
	\end{assumption}
	
	\begin{assumption}\label{corin}
		There is $\eta>0$ such that the limit 
		$$
		\Gamma(\alpha):=\lim_{n\to\infty}\frac{1}{n}\ln  \E e^{\alpha(\Delta(Y_1)+\ldots+\Delta(Y_n))}
		$$
		exists for all $\alpha\in (-\eta,\eta)$ and $\Gamma$ is continuously differentiable on $(-\eta,\eta)$.
	\end{assumption}
	
	\begin{assumption}\label{lineargrowth}
		There exist $K_1$, $K_2$, $K_3$ such that 
		$$
		|H(\theta,y)|\leq K_1 |\theta|+K_2|y|+K_3.
		$$
	\end{assumption}
	
	Note that Assumption \ref{lineargrowth} holds, in particular, if $H$ is Lipschitz-continuous.
	
	\begin{assumption}\label{momentum}  
	$Y_{0}$ is bounded, say, $|Y_{0}|\leq M$ a.s.
	%	$$
	%	E[e^{\beta|Y_0|^2}]<\infty
	%	$$
	%	holds for some $\beta>0$.
	\end{assumption}

\begin{remark}{\rm Boundedness of $Y_{0}$ could be relaxed 
%to assuming only $E[e^{\beta|Y_0|^2}]<\infty${}
%for some $\beta>0$. This relaxation leads to a weaker rate estimate through 
at the price of rather tedious technicalities hence
we prefer not to treat this here.} 
\end{remark}	

	It turns out that the law of $\theta_n$ tends to a limit as $n\to\infty$ and ergodic
	averages converge to the expectation under the limit law.
	
	\begin{theorem}\label{guar} Let $\la>0$ be small enough. Under Assumptions \ref{immeasurable}, \ref{corin},
		\ref{lineargrowth} and \ref{momentum},
		there exists a probability law $\mu(\lambda)$, independent of the initial value, such that
		$$
		\dtv (\mathrm{Law}(\theta_n),\mu(\lambda))\leq c_1 e^{-c_2 n^{1/3}},
		$$
		for some $c_1,c_2>0$ (which also depend on $\lambda$).
		Moreover, for arbitrary bounded measurable $\Phi:\R^d\to\R$,
		$$
		\frac{\Phi(\theta_0)+\ldots+\Phi(\theta_{n-1})}{n}\to \int_{\R^d}\Phi(z)\, \mu(\lambda)(dz),
		$$
		as $n\to\infty$ in $L^p$, for all $p\geq 1$. 
	\end{theorem}
	
	\begin{remark}{\rm The convergence rates given by the above theorem 
			are not sharp enough for practical purposes. However,
			Theorem \ref{guar} provides a \emph{universal ergodic property} for
			the stochastic gradient Langevin dynamics, irrespective of dependencies in the data stream
			(as long as they satisfy Assumption \ref{corin}). No result of this calibre
			has heretofore been available in the related literature.
			
			It is a natural question, how far $\mu(\lambda)$ is from the target probability $\pi$.{}
			It follows from \cite{5author} that, under suitable assumptions, the Wasserstein-$1$
			distance of $\mathrm{Law}(\theta_{n})$ from $\pi$ is of the order $\sqrt{\lambda}$, uniformly in
			$n$. Hence the same is true for $\mu(\lambda)$ and $\pi$.}  
	\end{remark}
	
	\begin{proof}[Proof of Theorem \ref{guar}.]
		Choose $V(\theta):=|\theta|^2$, $\theta\in\R^d$ and define
		\begin{equation*}
			Q(y,\theta,A):=\P (\theta-\la H(\theta,y)+\sqrt{\la}\xi_1\in A), 
		\end{equation*}
for all $y\in\mathcal{Y}:=\R^m$, $\theta\in\mathcal{X}:=\R^d$ and $A\in\mathcal{B}:=\B (\R^d)$.		{}
Noting that $\xi_{1}$ has mean zero, we have
		\begin{align*}
			[Q(y)V](\theta) &= \E [V(\theta-\la H(\theta,y)+\sqrt{\la}\xi_1)]\\
			&= \la \E|\xi_1|^2 + \la^2 |H(\theta,y)|^2+|\theta|^2-2\la \langle \theta,H(\theta,y)\rangle\\
			&\leq \la (d+2b) + 3\la^2[K_1^2|\theta|^2 +  K_2^2|y|^2+K_3^2]+(1-2\la \Delta(y)) |\theta|^2
		\end{align*} 
		so Assumption \ref{as:drift} holds with $K(y):=\la (d+2b) + 3\la^2 K_3^2+3\la^2 K_2^2|y|^2$, 
		$\ga(y):=1+3\la^2 K_1^2-2\la\Delta(y)$. Note that, due to the boundedness of $\Delta$,{}
		$\ga(y)\geq 0$ for all $y$ for $\la$ small enough, in fact $\ga(y)\geq \tilde{\ga}>0$
		for some $\tilde{\ga}$.
		By Assumption \ref{immeasurable}, for $\bar\la$ small enough, $\E[3\bar\la K_1^2-2\Delta(Y_0)]<0$. 
		
		Arguments similar to those in the preceeding section show that, when $\la\leq\bar{\la}$ is small enough,
		$$
		\limsup_{n\to\infty}\frac{1}{n}\ln E[e^{\la\sum_{j=1}^n [3\bar\la K_1^2-2\Delta(Y_j)]}]<0,
		$$
		but then also
		$$
		\limsup_{n\to\infty}\frac{1}{n}\ln E[e^{\la\sum_{j=1}^n [3\la K_1^2-2\Delta(Y_j)]}]<0.
		$$
		Noting $1+x\leq e^x$ this implies
		\begin{align*}
		\bar{\ga}:=\limsup_{n\to\infty}E^{1/n}[K(Y_0)\ga(Y_1)\ldots\ga(Y_n)]\le&\\
		\limsup_{n\to\infty}E^{1/n}[\ga(Y_1)\ldots\ga(Y_n)]
		\sqrt[n]{\la (d+2b) + 3\la^2 K_3^2+3\la^2 K_2^2 M^2}=&\\
		\limsup_{n\to\infty}E^{1/n}[\ga(Y_1)\ldots\ga(Y_n)]<&1
		\end{align*}
		hence Assumption \ref{as:LT} also holds.
		
		Let $0<\eps<1/\bar{\ga}^{1/2}-1$, $R(y):=\frac{2 K(y)}{\eps \ga (y)}$, define $C(y):=\{\theta\in\R^d: |\theta|^2\leq R(y)\}$
		and set 
		$$
		\ka (y,A):=\frac{\mathrm{Leb}(C(y)\cap A)}{\mathrm{Leb}(C(y))},\quad A\in\B (\R^d).
		$$ 
		Denoting $f(\theta):=\exp\{-|\theta|^2/2\}/(2\pi)^{d/2}$, $\theta\in\R^d$,
		for each $y\in\R^m$, $|y|\le M$, $\theta\in C(y)$ and $A\in\B (\R^d)$
		\begin{align*}
			  Q(y,\theta, A)&=\P(\theta-\la H(\theta,y)+\sqrt{\la}\xi_1\in A)
			  \ge \P(\theta-\la H(\theta,y)+\sqrt{\la}\xi_1\in C(y)\cap A)\\
			&\ge \int_{\R^d} \ind_{\theta-\la H(\theta,y)+w\sqrt{\la} \in C(y)\cap A}\, f\left(w\right)\, \dint w\\
			&= \frac{1}{\la^{d/2}}\int_{C(y)\cap A} f\left(\frac{z-\theta+\la H(\theta,y)}{\sqrt{\la}}\right)\, \dint z\\
			&\ge \frac{\mathrm{Leb}(C(y))}{(2\pi\la)^{d/2}}
			\exp \left(-\max_{z\in C(y)}\frac{|z-\theta+\la H(\theta,y)|^2}{2\la}\right)\ka (y,A).
		\end{align*}
		Note that, for $|y|\le M$ and $\theta, z\in C(y)$, we have
		\begin{align*}
			\frac{1}{2\la}|z-\theta+\la H(\theta,y)|^2
			&\le
			\frac{(2+\la K_1)^2}{\la}\frac{2K(y)}{\eps \ga(y)} + \la (K_2 M+K_3)^2 \\
			&\le
			\frac{2(2+\la K_1)^2\left[d+2b+3\la (K_3^2+K_2^2 M^2)\right]}{\eps\tilde{\ga}}
			+\la (K_2 M + K_3)^2.
		\end{align*}
		Clearly, we can choose $\lambda$ small enough such that
		\begin{equation*}
			\frac{1}{2\la}|z-\theta+\la H(\theta,y)|^2
			\le\frac{9(d+2b)}{\eps\tilde{\ga}}+1. 
		\end{equation*}
		According to our previous estimate for $Q(y,\theta, A)$, for $\la$ small enough,
		we have
		\begin{align*}
			 Q(y,\theta, A) &\ge \frac{\mathrm{Leb}(C(y))}
			 {(2\pi\la)^{d/2}}\exp\left({-\frac{9(d+2b)}{\eps\tilde{\ga}}-1}\right)\ka (y,A) \\
			 &\ge \tilde{c} e^{-\hat{c}/\eps}\ka (y,A)
		\end{align*}
		for suitable $\tilde{c},\hat{c}>0$ depending on $b,d,M$ and $\sup_{y\in\R^m}|\Delta (y)|$,
		%		Obviously, there exists $0<\eps<1/\bar{\ga}^{1/2}-1$ such that $\tilde{c} e^{-\hat{c}/\eps}<1$ 
		which proves that Assumption \ref{as:minor} and \ref{as:myas} hold with $\alpha:=1-\tilde{c} e^{-\hat{c}/\eps}$. 
		We thus get that the claimed convergence rate holds by Theorem \ref{thm:TV} and Remark \ref{simplify}.
\end{proof}

%%%%%%%%%%%%%%%%%%%%%%%%%%%%%%%%%%%%%%%%%%%%%%%%%%%
%%%%%%%%%%%%%%%%%%%%%%%%%%%%%%%%%%%%%%%%%%%%%%%%%%%
%%%%%%%%%%%%%%%%%%%%%%%%%%%%%%%%%%%%%%%%%%%%%%%%%%%

\section{Linear systems in a random environment}\label{sec:lin}

A popular class of examples where we can apply the results of Section \ref{sec:ram} is that of linear systems.
Fix integers $m,d\geq 1$ and let $A,B:\R^m\to\R^{d\times d}$ be
measurable functions. Operator norm of a matrix $M\in\R^{d\times d}$ will be denoted
$||| M|||$. Let $Y_t\in \mathcal{Y}:=\R^m$, $t\in\Z$ be a stationary process.
We consider the process $X_t\in\R^d$ obeying the linear dynamics
\begin{equation}\label{linearsystem}
X_{t+1}:=A(Y_t)X_t+B(Y_t)\eps_{t+1},\ t\in\N,
\end{equation}
where $\eps_t\in\R^d$, $t\geq 1$ is an i.i.d.\ sequence, independent of $(Y_t)_{t\in\Z}$.
Let $X_{0}=x_{0}$ with some constant $x_{0}\in\R^{d}$.

For simplicity we stay with the case of square matrices $A(\cdot)$, $B(\cdot)$. More general linear
systems could be treated along similar lines under suitable controllability conditions but at the price
of considerable complications.

\begin{assumption}\label{dyni} The functions $A,B$ are bounded; $B(y)$, $A(y)$ are invertible for all $y\in\R^m$ 
	such that $\sup_{y\in\R^m}\left(|||A(y)^{-1}|||+||| B(y)^{-1}|||\right)<\infty$; 
	$\eps_0$ has a density $f:\R^d\to\R_+$ 
	with respect to the $d$-dimensional Lebesgue measure which
	is a.s.\ bounded away from $0$ on compact sets; $E|\eps_{0}|<\infty$. 
\end{assumption}

\begin{assumption}\label{contri}
	There is an integer $p\geq 1$ such that
	\begin{equation}\label{stab}
	E\left[\ln||| A(Y_p)A(Y_{p-1})\ldots A(Y_1)|||\right]<0.
	\end{equation}
	Furthermore, there is $\eta>0$ and a differentiable function $\Gamma:(-\eta,\eta)\to\R$ such that
	\begin{equation}\label{eldepe}
	\lim_{n\to\infty} \frac{1}{n}\ln E\prod_{i=1}^n ||| A(Y_{ip})A(Y_{ip-1})\ldots A(Y_{(i-1)p+1})|||^{\alpha}=\Gamma(\alpha),\ 
	\alpha\in (-\eta,\eta).
	\end{equation}
\end{assumption}

Note that when $A$ is a constant matrix, \eqref{stab} boils down to requiring $|||A^p|||<1$ for some $p\geq 1$ which is equivalent 
to the spectral radius of $A$ being smaller than $1$. This shows that \eqref{stab} is
reasonable to assume. Condition 
\eqref{eldepe} is again a G\"artner-Ellis type condition
in the spirit of Assumption \ref{as:YLargeDev} above. For Markovian $Y$ sufficient conditions for
its fulfillment can be deduced from e.g.\ \cite{KM1,KM2}.

\begin{theorem}
	Under Assumptions \ref{dyni} and \ref{contri}, Theorem \ref{thm:gen} applies to the system \eqref{linearsystem}.	
\end{theorem}
\begin{proof}
	We first verify the drift condition. Choose $V(x):=|x|$, $x\in\R^{d}$.
	Define $\ga(y_{1},\ldots,y_{p}):=||| A(y_p)A(y_{p-1})\ldots A(y_1)|||$, $y_{1},\ldots,y_{p}\in\R^{m}$.
	Note that, by Assumption \ref{dyni}, there is $\bar{\ga}>0$ such that $\ga(y_{1},\ldots,y_{p})\geq \bar{\ga}$.
	
	Let $M\geq 1$ denote a bound for both $|||A(\cdot)|||$ and $|||B(\cdot)|||$. Now notice that
	$$
	[Q(y_{p})\ldots Q(y_{1})]V(x)\leq \ga(y_{1},\ldots,y_{p})V(x)+ pM^{p}\E|\eps_{0}|.
	$$
	Furthermore, clearly
	$$
	Q(y)V(x)\leq MV(x)+ M\E|\eps_{0}|,
	$$	
	so we have verified the drift conditions with $K:=pM^{p}\E|\eps_{0}|$.
	
	To check the contractivity condition, notice that
	\begin{eqnarray*}
		& & \frac{1}{n}\ln \E\prod_{i=1}^n ||| A(Y_{ip})A(Y_{ip-1})\ldots A(Y_{(i-1)p+1})|||^{\alpha}\\
		&=& \frac{1}{n}\ln \E\exp\left\{\alpha \sum_{i=1}^{n}\ln||| A(Y_{ip})A(Y_{ip-1})\ldots A(Y_{(i-1)p+1})|||\right\}
	\end{eqnarray*}
	and the summands are bounded by Assumption \ref{dyni}.
	Now from \eqref{stab} and \eqref{eldepe} by an argument identical to that of Lemma \ref{lem:queuing:DriftAndLT} 
	we can deduce that
	\begin{equation}
	\limsup_{n\to\infty} \E^{1/n}\left(\prod_{i=1}^{n}\ga (Y_{(i-1)p+1},\ldots,Y_{ip})\right) < 1.	
	\end{equation}
	
	Now let us turn to the minorization condition. 
	Let $U_{r}$ denote the closed ball of radius $r>0$ around the origin in $\R^{d}$.
	Let $F>0$ be fixed and let $x\in U_{F}$ be
	arbitrary. Let $A\in\mathcal{B}(U_{1})$ and $L>0$ be arbitrary. 
	Taking $x\in\R^{d}$ and $y_{p},\ldots,y_{1}\in\R^{m}$,
	\begin{eqnarray*}
		& & \P\left(A(Y_{p})\cdots A(Y_{1})x+\sum_{i=1}^{p}A(Y_{p})\cdots A(Y_{i+1})B(Y_{i})\eps_{i}\in A\right)\\
		&\geq& \E\left[1_{\{\sup|A(y_{p})\cdots A(y_{1})x+\sum_{i=1}^{p-1}A(y_{p})\cdots A(y_{i+1})B(y_{i})\eps_{i}|\leq L\}}\right.
		\\ &\times&
		\left.\inf_{w\in U_{L},y\in\R^{m}}\P(w+B(y)\eps_{p}\in A|\sigma(\eps_{1},\ldots,\eps_{p-1}))\right]\\
		&= &  \P\left(\sup\left|A(y_{p})\cdots A(y_{1})x+\sum_{i=1}^{p-1}A(y_{p})\cdots A(y_{i+1})B(y_{i})\eps_{i}\right|\leq L\right)
		\\
		&\times&{}
		\min_{w\in U_{L},y\in\R^{m}}\P(w+B(y)\eps_{p}\in A) 
	\end{eqnarray*}
	by independence of the $\eps_{i}$. Here the sup is taken over all $y_{p},\ldots,y_{1}\in\R^{m}$.
	Choose $L$ so large that $P(\sup|A(y_{p})\cdots A(y_{1})x+\sum_{i=1}^{p-1}A(y_{p})\cdots A(y_{i+1})B(y_{i})\eps_{i}|\leq L)
	\geq 1/2$. This is clearly possible since $A$, $B$ are bounded and $x\in U_{F}$.
	
	Now notice that by Assumption \ref{dyni}, in particular, by $\sup_{y}|||B(y)^{-1}|||<\infty$, 
	the density functions (w.r.t.\ the Lebesgue measure) of the random variables
	$B(y)\eps_{p}$, $y\in\R^{m}$ are bounded away from $0$ on $U_{L+1}$. But this means
	that $$
	\min_{w\in U_{L},y\in\R^{m}}P(w+B(y)\eps_{p}\in A)\geq \bar{\alpha}\mathrm{Leb}(A)
	$$ for some $\bar{\alpha}>0$, independent of $A$. 
	Applying these observations in the particular case $F:=\frac{2K}{\epsilon\bar{\ga}}$ we obtain the minorization
	condition and Theorem \ref{thm:gen} indeed applies.
\end{proof}

%%%%%%%%%%%%%%%%%%%%%%%%%%%%%%%%%%%%%%%%%%%%%%%%%%%%%%%%
%%%%%%%%%%%%%%%%%%%%%%%%%%%%%%%%%%%%%%%%%%%%%%%%%%%%%%%%
%%%%%%%%%%%%%%%%%%%%%%%%%%%%%%%%%%%%%%%%%%%%%%%%%%%%%%%%
	
	\section{Proofs}\label{sec:proofs}
	
	Now we proceed to the proofs of Theorems \ref{thm:TV} and \ref{thm:LLN}. All the assumptions of those 
	results are supposed to hold throughout this section. In order to make our explanation understandable for 
	the largest possible audience, we present here the key steps of the proofs together with the 
	fundamental ideas behind them. Lemmas crucial to proving our main theorems are mentioned here, whereas 
	the role of the more technical ones are clarified in the text body.
	
	First, we introduce a representation for the process $X_{t}$ using random maps depending on the environment. 
	Furthermore, we show that these maps are constant on ``small sets'' with
	positive probability, where these ``small sets'' and this probability in question are 
	determined by the instantaneous value of $Y$ (Lemma \ref{lem:T}). Next, 
	we freeze the environment and estimate the probability of the event that two copies 
	of $X_t$, $t\in\N$ starting from two different random states and driven by the same fixed trajectory of 
	$Y_t$, $t\in\Z$ become coupled after $N^3$ steps (Lemma \ref{lem:central}). In order to prove this, 
	we show that with large probability, in $N^3$ steps, both of the representatives visits simultaneously 
	the same small set at least $N$ times and by Lemma \ref{lem:T}, after each visit, they are mapped to 
	the same state with positive probability.
	
	Using the transportation cost characterization of $\dtv$, the strong stationarity of $Y_t$, $t\in\Z$ and our 
	results on the coupling probability, we show that $\mu_{n}=\mathrm{Law}(X_{n})$, $n\in\mathbb{N}$ is a Cauchy sequence in 
	the complete metric space $(\M_1,\dtv)$ which proves Theorem \ref{thm:TV}.
	
	Our approach to the ergodic theorem for $X$ relies on Birkhoff's ergodic theorem and
	the $L$-mixing property of a certain auxiliary Markov chain. It turned out in \cite{rasonyi2018} 
	that $L$-mixing is particularly well-adapted to
	Markov chains, even when they are inhomogeneous (and for us this is the crucial
	point). The main ideas of arguments in Subsection \ref{sec:proofs:Zt} go back to 
	\cite{rasonyi2018}. Strong stationarity of $Y$ together with our estimate for the coupling probability 
	are extensively used at this point.
	
	The majority of technical lemmas 
	(e.g.\ Lemmas \ref{lem:cut}, \ref{lem:dtv} and \ref{lem:WtypicalY}) are created to identify 
	finite dimensional subsets of $\Y^\Z$
	such that the probability of coupling on trajectories coming from these sets
	and the probability of finding paths of $Y$ in these sets are large enough. These arguments strongly 
	rely on the stationarity of $Y$ as well as the long-time contractivity of the chain.

	\subsection{Preliminary lemmas and notations}\label{sec:proofs:lem}
	
	For $R>0$, we denote by $\mathbf{c}(R)$ the set of mappings from $\X$ into $\X$ whose restriction to 
	$V^{-1}([0,R])$ is constant. $\eps>0$ and $R(y)$ will be as in Assumption \ref{as:minor}. 

	Representing Markov chains on Polish spaces by iterated random maps
	is a standard construction, see e.g.\ \cite{bw}. A similar 
	representation for $Q$ is shown in Lemma \ref{lem:T} below which will play a crucial role in the proofs.
	It is a variant of Lemma 6.1 in \cite{rasonyi2018} in a somewhat more general setting.
	% but nothing special is assumed about the structure of $\Y$.	
	% T map representation
	\begin{lemma}\label{lem:T}
		There exists a sequence of measurable functions $T_t:\Y\times\X\times\Omega\to\X$, $t\in\Z$ such that
		\begin{equation*}
		\P\left(\{\omega\in\Omega\mid T_t (y,x,\cdot)\in A\}\right) = Q(y,x,A),
		\end{equation*}
		for all $t\in\Z$, $y\in\Y$, $x\in\X$, $A\in\B$ and there are events $J_t(y)\in\F$, for all $t\in\Z$, $y\in\Y$ such that
		\begin{equation}\label{eq:eventJ}
		J_t(y)\subset\{\omega\in\Omega\mid T_t(y,\cdot,\omega)\in \mathbf{c}(R(y))\}
		\text{ and }\P (J_t(y))\ge 1-\alpha (y)
		\end{equation}
		Furthermore, the sigma-algebras $\sigma (T_t (y,x,\cdot),\,x\in\X,\,y\in\Y)$, $t\in\Z$ are independent.	
	\end{lemma}
	\begin{proof}
		We follow the proof of Lemma 6.1 in \cite{rasonyi2018}.	
		So, let $U_n$ and $\eps_n$, $n\in\Z$ be sequences of i.i.d.\ uniform random variables on $[0,1]$ 
		independent of each other.
		Without loss of generality, we may assume that $(U_t, \eps_t)$, $t\in\Z$ independent of 
		$Y_t$, $t\in\Z$. The case of countable $\X$ is easy hence omitted. In the case of
		$\X$ uncountable we can also assume (by the Borel isomorphism theorem) that $\X=\R$ and 
		$\B(\R)$ is the standard Borel $\sigma$-algebra of $\R$.
		
		Easily seen that, if for some $y\in\Y$ $\alpha (y)=0$, then by Assumption \ref{as:minor}, for $A\in\B$ and $x\in V^{-1}([0,R(y)])$, $Q(y,x,A) \ge \ka (y,A)$ and $Q(y,x,\X\setminus A) \ge \ka (y,\X\setminus A)$ hold at the same time and thus we have
		\begin{equation*}
		 Q(y,\cdot,A)\vert_{V^{-1}([0,R(y)])} = \ka (y,A),\quad A\in\B.
		\end{equation*}
		
		For $y\in\Y$, $x\in\X$ and $A\in\B(\R)$, let
		\begin{align*}
			q(y,x,A) := \begin{cases}
				\frac{1}{\alpha (y)}\left[
			Q(y,x,A)-(1-\alpha (y))\ka (y,A)
			\right]\ind_{V(x)\le R(y)} \\ 
			+ Q(y,x,A)\ind_{V(x)>R(y)} & \text{ if } \alpha (y)\ne 0\\[1.2em]
			Q(y,x,A)\ind_{V(x)>R(y)} & \text{ if } \alpha (y)=0
			\end{cases}
		\end{align*}
		and define
		\begin{equation*}
			T_t (y,x,\omega) =
				\ka^{-1}(y,\eps_t)\ind_{U_t\le 1-\alpha (y)}\ind_{V(x)\le R(y)}
			+
			q^{-1}(y,x,\eps_t) \left(1-\ind_{U_t\le 1-\alpha (y)}\ind_{V(x)\le R(y)}\right)
		\end{equation*}
		where 
		\begin{align*}
			\ka^{-1}(y,z) &:= \inf \{r\in\Q\mid \ka (y,(-\infty,r])\ge z\} \\
			q^{-1}(y,x,z) &:=  \inf \{r\in\Q\mid q(y,x,(-\infty,r])\ge z\},\,	
		\end{align*}
		$z\in\R$ are the pseudoinverses of the corresponding cumulative distribution functions.
		
		Obviously, $x\mapsto T_t(y,x,\omega)$ is constant on $V^{-1}([0,R(y)])$ whenever $U_t\le 1-\alpha (y)$, this implies \eqref{eq:eventJ} with 
		$J_t (y):=\{\omega \mid U_t (\omega)\le 1-\alpha (y)\}$. Furthermore, for all $r\in\R$, $t\in\Z$ and for any fixed $y\in\Y$ and $x\in\X$
		\begin{align*}
			&\P\left(\{\omega\in\Omega\mid T_t (y,x,\cdot)\le r\}\right) = \ind_{V(x)>R(y)} \P \left(q^{-1}(y,x,\eps_t)\le r \right) \\
			&+\ind_{V(x)\le R(y)} \left[\alpha (y) \P \left(q^{-1}(y,x,\eps_t)\le r \right) + 
			(1-\alpha (y)) \P \left(\ka^{-1}(y,\eps_t)\le r \right)
			\right].
		\end{align*}
		By the definition of the pseudoinverse, we can write
		\begin{align*}
			\P \left(\ka^{-1}(y,\eps_t)\le r \right) &= 
			\P \left(\ka(y,(-\infty,r'])\ge \eps_t,\, r'\in\Q\cap (r,\infty) \right) \\
			&= \P \left(\ka(y,(-\infty,r])\ge \eps_t\right) = \ka(y,(-\infty,r])
		\end{align*}
		and similarly
		\begin{equation*}
			\P \left(q^{-1}(y,x,\eps_t)\le r \right) = q(y,x,(-\infty,r])
		\end{equation*}
		hence
		\begin{align*}
			&\P\left(\{\omega\in\Omega\mid T_t (y,x,\cdot)\le r\}\right) =  Q(y,x,(-\infty,r])
		\end{align*}
		as we desired.
		
		It remains only to show that $T_t$ is measurable with respect to sigma algebras $\A\otimes\B (\R)\otimes \sigma \left(\{U_t,\eps_t\mid t\in\Z\}\right)$ and $\B (\R)$. Indeed, $T_t$ is a composition of measurable functions. The claimed independence of the sigma-algebras clearly holds too.
	\end{proof}
	
	% Auxiliary processes

	We drop the dependence of the mappings $T_t$ on $\omega$ in the notation and will simply write $T_t(y)\,x:=T_t(x,y,\cdot)$. 
	For $s\in\Z$ and $x\in\X$, 
	define the family of auxiliary processes
	\begin{equation}\label{eq:aux}
	Z_{s,s}^{x,\mathbf{y}} = x, \quad Z_{s,t+1}^{x,\mathbf{y}} = T_{t+1}(y_t)Z_{s,t}^{x,\mathbf{y}},\quad t\ge s,  
	\end{equation}
	where $\mathbf{y} = (\ldots,y_{-1},y_0,y_1,\ldots)\in\Y^\Z$ is a fixed trajectory.
	Let $\G_t := \sigma (\eps_i, U_i,\, i\le t)$ and $\G_t^+ := \sigma (\eps_i, U_i,\, i>t)$,
	$t\in\Z$. Clearly, $\G_t$ is independent of $\G_t^+$ and $Z_{s,t}^{x,\mathbf{y}}$ is adapted to $\G_t$ moreover the process $Z_{s,t}^{x, \mathbf{y}}$, $t\ge s$ heavily
	depends on the choice of $\mathbf{y}$. These processes follow the dynamics of $X$ with the environment being
	``frozen''.	
	
	\smallskip
	Recall that $S:\Y^\Z\to\Y^\Z$ stands for the usual \emph{left shift operation} i.e. 
	\begin{equation}\label{eq:shift}
	(S\mathbf{y})_j=\mathbf{y}_{j+1},\,j\in\Z.
	\end{equation}
	
	In the, next lemma we introduce subsets of $\Y^\Z$ such that for fixed $M,N\in\N^+$, contractivity 
	is established on every consecutive $N^{1/M}$ long pieces of $[1,N]$.
	Furthermore, we prove that $\mathbf{Y}$ and its shifted versions fall into these sets
	with large probability.
	% Cutting lemma
	
	\begin{lemma}\label{lem:cut}
		For $M,N\in\N^+$ and $\la\in (0,1)$, we define the sets
		\begin{equation*}
			B_{N,M}^\la := \left\{
			\mathbf{y}\in\Y^\Z \middle| \prod_{l=1}^{\lfrf{N^{1/M}}}
			\ga \left(
			y_{k\lfrf{N^{1/M}}+l}
			\right) < \bar{\ga}^{\la \lfrf{N^{1/M}}}, \, k=0,1,\ldots,\lfrf{N^{1/M}}^{M-1}-1
			\right\}. 
		\end{equation*}
		Then $\mathbf{Y}$ and its shifted copies fall into $B_{N,M}^\la$ with large probability. More precisely, there exist $c,\nu >0$ such that
		\begin{align*}
			(\forall k\in\Z)\quad\P \left((S^k\mathbf{Y})\in B_{N,M}^\la \right) &\ge 1-c N^{1-1/M}e^{-\nu N^{1/M}} \\
			\P \left((S^k\mathbf{Y})\in B_{N,M}^\la,\,k=0,\ldots n-1 \right) &\ge 1-c n N^{1-1/M}e^{-\nu N^{1/M}} 
		\end{align*}
		holds.
	\end{lemma}
	\begin{proof}
		The second inequality easily follows from the first one. By the union bound and the strong stationarity of $Y_t$, $t\in\Z$, we can write
		\begin{align*}
			\P \left(\bigcup_{k=0}^{n-1} (S^k\mathbf{Y})\notin B_{N,M}^\la\right)	&\le 
			\sum_{k=0}^{n-1}\P \left( (S^k\mathbf{Y})\notin B_{N,M}^\la\right) =
			n\P \left(\mathbf{Y}\notin B_{N,M}^\la\right) \le c n N^{1-1/M}e^{-\nu N^{1/M}}.
		\end{align*}
		
		In order to prove the first inequality, we use the union bound and the strong stationarity of $Y_t$, $t\in\Z$ again.
		\begin{align*}
			\P \left(\mathbf{Y}\notin B_{N,M}^\la\right) &\le \sum_{k=0}^{\lfrf{N^{1/M}}^{M-1}-1}
			\P \left( \prod_{l=1}^{\lfrf{N^{1/M}}} \ga \left( Y_{k\lfrf{N^{1/M}}+l}\right)\ge\bar{\ga}^{\la N^{1/M}}\right) \\
			&\le N^{1-1/M}\P \left( \prod_{l=1}^{\lfrf{N^{1/M}}} \ga \left( Y_{l}\right)\ge\bar{\ga}^{\la N^{1/M}}\right) \le  N^{1-1/M} \frac{\E \left(\prod_{l=1}^{\lfrf{N^{1/M}}}\ga \left( Y_{l}\right)\right)}{\bar{\ga}^{\la N^{1/M}}}.
		\end{align*}
		By Assumption \ref{as:LT}, there exists $\tilde{c}>0$ such that
		$
			\E\left(K(Y_0)\prod_{t=1}^{n} \ga (Y_t)\right)\le \tilde{c}\bar{\ga}^{\frac{1+\la}{2}n}
		$, $n\in\N$ moreover $K(\cdot)\ge 1$. With this, we get
		\begin{equation*}
			\E \left(\prod_{l=1}^{\lfrf{N^{1/M}}}\ga \left( Y_{l}\right)\right) \le \tilde{c}
			\bar{\ga}^{\frac{1+\la}{2}\lfrf{N^{1/M}}} \le 
			\frac{\tilde{c}}{\bar{\ga}^{\frac{1+\la}{2}}}\bar{\ga}^{\frac{1+\la}{2}N^{1/M}}
		\end{equation*}
		and
		\begin{equation*}
		\P \left(\mathbf{Y}\notin B_{N,M}^\la\right) \le \frac{\tilde{c}}{\bar{\ga}^{\frac{1+\la}{2}}} N^{1-1/M}\bar{\ga}^{\frac{1-\la}{2}N^{1/M}}
		\end{equation*}
		thus the first inequality holds with $c=\tilde{c}/\sqrt{\bar{\ga}^{1+\la}}$
		and $\nu = -\frac{1-\la}{2}\log\bar{\ga}$.
		
	\end{proof}

	Let $N\in\N^+$ and $\mathbf{y}\in\Y^\Z$ be such that $S^t\mathbf{y}\in B_{\lfrf{N^{1/6}}^6,6}^\la$, $t=0,1,\ldots\lfrf{N^{1/6}}-1$. For
	$t\in\left\{t=1,\ldots \lfrf{N^{1/6}}^6-\lfrf{N^{1/6}}^3\right\}$,
	exists $q\in\left\{0,1,\ldots, \lfrf{N^{1/6}}^5-\lfrf{N^{1/6}}^2 \right\}$
	and $r\in\left\{0,1,\ldots,\lfrf{N^{1/6}}-1\right\}$ for which
	$t=q\lfrf{N^{1/6}}+r$ holds. So, for $k\in\left\{0,1,\ldots,\lfrf{N^{1/6}}^2-1\right\}$,
	we have
	\begin{align*}
		\prod_{l=1}^{\lfrf{N^{1/6}}} \ga \left(\left(S^t\mathbf{y}\right)_{k\lfrf{N^{1/6}}+l}\right)
		&=
		\prod_{l=1}^{\lfrf{N^{1/6}}} \ga \left(\mathbf{y}_{t+k\lfrf{N^{1/6}}+l}\right)
		=
		\prod_{l=1}^{\lfrf{N^{1/6}}} \ga \left(\mathbf{y}_{(k+q)\lfrf{N^{1/6}}+l+r}\right)
		\\
		&=
		\prod_{l=1}^{\lfrf{N^{1/6}}} \ga \left(\left(S^r\mathbf{y}\right)_{(k+q)\lfrf{N^{1/6}}+l}\right)
		<\bar{\ga}^{\la\lfrf{N^{1/6}}}
	\end{align*}
	hence $S^t\mathbf{y}\in B_{\lfrf{N^{1/6}}^3,3}^\la$.
	Thus we arrive at the following important remark.
	
	\begin{remark}\label{rem:six}
	{\rm If for some $N\in\N^+$ and $\mathbf{y}\in\Y^\Z$, 
	$S^t\mathbf{y}\in B_{\lfrf{N^{1/6}}^6,6}^\la$, $t=0,1,\ldots\lfrf{N^{1/6}}-1$, then
	$S^t\mathbf{y}\in B_{\lfrf{N^{1/6}}^3,3}^\la$, $t=1,\ldots \lfrf{N^{1/6}}^6-\lfrf{N^{1/6}}^3$ holds as well.}		
	\end{remark}

	The next lemma tells us about the consequences of the drift condition satisfied by 
	$Q(y_{k-1})\ldots Q(y_l)$, where $\mathbf{y}\in\Y^\Z$ and $k,l\in\Z$, $l<k$ are arbitrary and fixed.
	
	\begin{lemma}\label{lem:ita} % Iterated action of Q(y_0),...,Q(y_n)
	For $x\in\X$, $\mathbf{y}\in\Y^\Z$ and $k,l\in\Z$, $l<k$, we have	
	\begin{align*}
	\left[Q(y_{k-1})\ldots Q(y_{l}) V\right](x) &\le 
	V(x)\,\prod_{r=l}^{k-1}
	\ga (y_r)
	+ \sum_{r=l}^{k-1}K(y_r)\prod_{j=r+1}^{k-1}
	\ga (y_j).
	\end{align*}
\end{lemma}
\begin{proof}
	We prove by induction. Let $x\in\X$ and $l\in\Z$ be arbitrary and fixed. For $k=l+1$, 
	we have
	\begin{equation}
	\left[Q(y_{l})V\right](x) \le \ga (y_l)V(x) + K(y_l).
	\end{equation}
	which holds by Assumption \ref{as:drift}.
	
	\medskip\noindent
	\emph{Induction step:}
	Operators $V\mapsto [Q(y)V]$, $y\in\Y$ are linear, monotone and for $V\equiv 1$ $[Q(y)V]\equiv 1$, $y\in\Y$ hence by Assumption \ref{as:drift} we can write
	\begin{align*}
	\left[Q(y_{k})\ldots Q(y_{l}) V\right](x) = \left[Q(y_{k})\left(Q(y_{k-1})\ldots Q(y_{l})V\right)\right](x) & \le \\
	[Q(y_k )V](x)\,\prod_{r=l}^{k-1}
	\ga (y_r)
	+ \sum_{r=l}^{k-1}K(y_r)\prod_{j=r+1}^{k-1}
	\ga (y_j) & \le \\
	V(x)\,\prod_{r=l}^{k} \ga (y_r) + \sum_{r=l}^{k}K(y_r)\prod_{j=r+1}^{k} \ga (y_j)
	\end{align*}
	which completes the proof.  
\end{proof}
	
	% Estimation of E(V(Z_t))
%	\begin{lemma}\label{lem:EV}
%		For $N\in\N^+$, $\la\in (0,1)$ and for a $Z$ $\G_0$-measurable, $\X$-valued random variable, we define
%		\begin{equation*}
%			B_{N^3,3}^{\la, Z} := \left\{
%			\mathbf{y}\in B_{N^3,3}^\la\middle| \E \left[V\left(Z_{0,k N^2}^{Z,\mathbf{y}}\right)\right] < \E (V(Z))+\bar{\ga}^{-\la N},\,
%			k=0,1,\ldots, N-1
%			\right\}.
%		\end{equation*}
%		Then there exists $c>0$ such that $\P \left(\mathbf{Y}\in B_{N^3,3}^{\la, Z}\right)\ge 1-c \bar{\ga}^{\la N}$.
%	\end{lemma}
%	\begin{proof}
%		If $\mathbf{y}\in B_{N^3,3}^\la$, then
%		by Lemma \ref{lem:ita} and the tower rule, we have
%		\begin{align*}
%			\E \left[V\left(Z_{0,k N^2}^{Z,\mathbf{y}}\right)\right]
%			&=
%			\E \left(\left[Q(y_{kN^2-1})\ldots Q(y_0)\right]V(Z)\right) \\
%			&\le
%			\E (V(Z))\prod_{p=0}^{kN^2-1}\ga (y_p) + 
%			\sum_{p=0}^{kN^2-1} K(y_p)\prod_{j=p+1}^{kN^2-1}\ga (y_j) \\
%			&\le 
%			\E (V(Z))\bar{\ga}^{\la k N^2} + \sum_{q=0}^{kN-1}\sum_{r=0}^{N-1}
%			K(y_{qN+r})\prod_{j=qN+r+1}^{(q+1)N-1}\ga (y_j)\prod_{i=(q+1)N}^{kN^2-1}\ga (y_i)
%			\\
%			&\le
%			\E (V(Z))\bar{\ga}^{\la k N^2} + 
%			\sum_{q=0}^{kN-1} \bar{\ga}^{\la N(kN-(q+1))}
%			\left(
%			\sum_{r=0}^{N-1} K(y_{qN+r}) \prod_{j=qN+r+1}^{(q+1)N-1}\ga (y_j)
%			\right).
%		\end{align*}
%		
%		On the other hand, by Markov inequality, we can write
%		
%	\end{proof}
	
	Let $N\in\N^+$, $\la\in (1/2,1)$ be fixed and $P_1, P_2:\Omega\to\X$ arbitrary 
	$\G_0$-measurable random variables, which may depend on $\mathbf{y}$. Furthermore, 
	in the remaining part of this subsection, we assume that $\mathbf{y}\in B_{N^3,3}^\la$. Our purpose 
	will be to prove that, with a large probability 
	$Z_{0,N^3}^{P_1,\mathbf{y}}=Z_{0,N^3}^{P_2,\mathbf{y}}$ for $N$ large enough. In other words, 
	a coupling between the processes $Z_{0,N^3}^{P_1,\mathbf{y}}$ and $Z_{0,N^3}^{P_2,\mathbf{y}}$ is realized.
	
	First, we are going to prove that the process
	$\overline{Z}_t:=\left(Z_{0,t}^{P_1,\mathbf{y}}, Z_{0,t}^{P_2,\mathbf{y}}\right)$, $t\in\N$ visits the sets $\overline{D}(y_t)$ frequently enough, where
	\begin{equation}
	\overline{D}(y)=\left\{
	(x_1,x_2)\in\X^2\middle| V(x_1)+V(x_2)\le R(y)
	\right\},\quad y\in\Y.
	\end{equation}
	Let us define the successive visiting times
	\begin{equation}
	\sigma_0 := 0,\quad \sigma_{k+1}:=\min \{i>\sigma_k\mid \overline{Z}_i\in\overline{D}( y_i)\},\quad k\in\N
	\end{equation}
	that are obviously $\left(\G_t\right)_{t\in\N}$-stopping times.
	
	% Return times
	\begin{lemma}\label{lem:sigma}
		For the tail distribution of $\sigma_{N}$, we have
	\begin{equation*}
	\P \left( \sigma_N > N^3\right) <
	\bar{\ga}^{\left(\la-\frac{1}{2}\right)N^2}\frac{1-\sqrt{\bar{\ga}}}{2}
	\sum_{k=0}^{N-1}\E \left[V(Z_{0,kN^2+1}^{P_1,\mathbf{y}})+V(Z_{0,kN^2+1}^{P_2,\mathbf{y}})\right].
	\end{equation*}
	\end{lemma}
	\begin{proof}
		If $\sigma_{N} > N^3$, then exists $k\in\{0,\ldots,N-1\}$ for which $\overline{Z}_{kN^2+l}\notin\overline{D}(y_{kN^2+l})$, $l=1,\ldots,N^2$. 
		Thus we can write
		\begin{align}\label{eq:skt}
		\begin{split}
		&\P \left( \sigma_{N} > N^3\right) \le\\
		&\le \P \left( 
		\bigcup_{k=0}^{N-1}\bigcap_{l=1}^{N^2}
		\left\{\overline{Z}_{k N^2+l}\notin\overline{D}(y_{kN^2+l})\right\}
		\right) \le
		\sum_{k=0}^{N-1}\P \left( 
		\bigcap_{l=1}^{N^2}
		\left\{\overline{Z}_{k N^2+l}\notin\overline{D}(y_{k N^2+l})\right\}
		\right). 
		\end{split}
		\end{align}
		
		We estimate a general term of the latter sum. For typographical reasons, we will write $a:=kN^2$ and $b:=N^2$. By the tower rule, we have
		
		\begin{align*}
		\P \left( 
		\bigcap_{l=1}^{b}
		\left\{\overline{Z}_{a+l}\notin\overline{D}(y_{a+l})\right\}
		\right)
		= \E\left[
		\prod_{l=1}^{b}
		\ind_{\left\{\overline{Z}_{a+l}\notin\overline{D}(y_{a+l})\right\}}
		\right] = &\\
		\E\left[
		\E\left[
		\ind_{\left\{\overline{Z}_{a+b}\notin\overline{D}(y_{a+b})\right\}}
		\middle| \G_{a+b-1}
		\right]
		\prod_{l=1}^{b-1}
		\ind_{\left\{\overline{Z}_{a+l}\notin\overline{D}(y_{a+l})\right\}}
		\right] < & \\
		\E\left[
		\E\left[
		\frac{V(Z_{0,a+b}^{P_1,\mathbf{y}}) + V(Z_{0,a+b}^{P_2,\mathbf{y}})}{R(y_{a+b})}
		\middle| \G_{a+b-1}
		\right]\prod_{l=1}^{b-1}
		\ind_{\left\{\overline{Z}_{a+l}\notin\overline{D}(y_{a+l})\right\}}\right] &.
		\end{align*}
		By Assumption \ref{as:drift}, we can write
		\begin{align*}
		\E\left[V(Z_{0,a+b}^{P_1,\mathbf{y}}) + V(Z_{0,a+b}^{P_2,\mathbf{y}})
		\middle| \G_{a+b-1}
		\right] = & \\
		\E\left[V(T_{a+b}(y_{a+b-1})Z_{0,a+b-1}^{P_1,\mathbf{y}}) + V(T_{a+b}(y_{a+b-1})Z_{0,a+b-1}^{P_2,\mathbf{y}})
		\middle| \G_{a+b-1}
		\right] = & \\
		[Q(y_{a+b-1})V](Z_{0,a+b-1}^{P_1,\mathbf{y}}) + [Q(y_{a+b-1})V](Z_{0,a+b-1}^{P_2,\mathbf{y}})  \le & \\
		\ga (y_{a+b-1})
		\left[
		V(Z_{0,a+b-1}^{P_1,\mathbf{y}})+V(Z_{0,a+b-1}^{P_2,\mathbf{y}})
		\right]+2 K(y_{a+b-1}).
		\end{align*}
		On the other hand, if $\overline{Z}_{a+b-1}\notin\overline{D}(y_{a+b-1})$,
		then 
		\begin{equation*}
			V(Z_{0,a+b-1}^{P_1,\mathbf{y}})+V(Z_{0,a+b-1}^{P_2,\mathbf{y}})>R(y_{a+b-1})
		\end{equation*}
		which immediately implies that
		\begin{equation*}
		2 K(y_{a+b-1})<\eps\ga (y_{a+b-1}) \left(V(Z_{0,a+b-1}^{P_1,\mathbf{y}})+V(Z_{0,a+b-1}^{P_2,\mathbf{y}})\right).
		\end{equation*}
		Recall that $0<\eps<1/\sqrt{\bar{\ga}}$ thus we have
		\begin{align*}
		\E\left[V(Z_{0,a+b-1}^{P_1,\mathbf{y}})+V(Z_{0,a+b-1}^{P_2,\mathbf{y}})
		\middle| \G_{a+b-1}
		\right] \ind_{\left\{\overline{Z}_{a+b-1}\notin\overline{D}( y_{a+b-1})\right\}}  <& \\
		\frac{\ga (y_{a+b-1})}{\sqrt{\bar{\ga}}}
		\left[
		V(Z_{0,a+b-1}^{P_1,\mathbf{y}})+V(Z_{0,a+b-1}^{P_2,\mathbf{y}})
		\right] &.
		\end{align*}
		This argument can clearly be iterated and leads to
		\begin{align*}
		\P \left( 
		\bigcap_{l=1}^{b}
		\left\{\overline{Z}_{a+l}\notin\overline{D}(y_{a+l})\right\}
		\right)
		< \frac{\prod_{l=1}^{b-1}\ga (y_{a+l})}{\sqrt{\bar{\ga}^{b-1}} R(y_{a+b})} \times 
		\E \left[V(Z_{0,a+1}^{P_1,\mathbf{y}})+V(Z_{0,a+1}^{P_2,\mathbf{y}})\right]. 
		\end{align*}
		
		Taking into account that $\mathbf{y}\in B_{N^3,3}^\la$, $R(y_{a+b})=\frac{2K(y_{a+b})}{\eps \ga (y_{a+b})}$ and $K(\cdot)\ge 1$, 
		hence we can write
		\begin{align*}
		\prod_{l=1}^{b}\ga (y_{a+l}) = \prod_{l=1}^{N^2}\ga (y_{kN^2+l}) =
		\prod_{l=0}^{N-1}\prod_{j=1}^{N}
		\ga (
		y_{(kN+l)N+j
		}
		) < \bar{\ga}^{\la N^2} 
		\end{align*}
		moreover
		\begin{align*}
		\frac{\prod_{l=1}^{b-1}\ga (y_{a+l})}{\sqrt{\bar{\ga}^{b-1}} R(y_{a+b})} \le
		\frac{1-\sqrt{\bar{\ga}}}{2}\bar{\ga}^{\left(\la-\frac{1}{2}\right)N^2}.
		\end{align*}

		Finally, we sum up for $k=0,1,\ldots,N-1$ and get
		\begin{equation*}
		\P \left( \sigma_N > N^3\right) <
		\bar{\ga}^{\left(\la-\frac{1}{2}\right)N^2}\frac{1-\sqrt{\bar{\ga}}}{2}
		\sum_{k=0}^{N-1}\E \left[V(Z_{0,kN^2+1}^{P_1,\mathbf{y}})+V(Z_{0,kN^2+1}^{P_2,\mathbf{y}})\right]
		\end{equation*}
		which completes the proof.
	\end{proof}
	
	Now we are in the position to estimate the probability of coupling between 
	$Z_{0,N^3}^{P_1,\mathbf{y}}$ and $Z_{0,N^3}^{P_2,\mathbf{y}}$.
	% Realization of coupling
	\begin{lemma}\label{lem:central}
		We have
		\begin{align*}
			\P\left(Z_{0,N^3}^{P_1,\mathbf{y}}\ne Z_{0,N^3}^{P_2,\mathbf{y}}\right) 
			&\le 
			\left(\max_{0\le k<N^3}\alpha (y_k)\right)^{N-1} \\
			&+
			\bar{\ga}^{\left(\la-\frac{1}{2}\right)N^2}\frac{1-\sqrt{\bar{\ga}}}{2}
			\sum_{k=0}^{N-1}\E \left[V(Z_{0,kN^2+1}^{P_1,\mathbf{y}})+V(Z_{0,kN^2+1}^{P_2,\mathbf{y}})\right].
		\end{align*}
	\end{lemma}
	\begin{proof}
		For typographical reasons, we will write $\sigma (N)$ instead of $\sigma_{N}$ in this proof.
		Recall that, $T_t (y):\X\to\X$ is constant on $V^{-1}([0,R( y)])$, $t\in\Z$, $y\in\Y$ with probability at least $1-\alpha (y)$. By the definition of $\overline{D}(\cdot)$,
		$\overline{Z}_t\in\overline{D}(y_t)$ implies that $Z_{0,t}^{P_i,\mathbf{y}}\in V^{-1}([0,R(y_t)])$ for $i=1,2$, moreover $Z_{0,t}^{P_1,\mathbf{y}}=Z_{0,t}^{P_2,\mathbf{y}}$
		with probability at least $1-\alpha (y_t)$.
		
		\medskip
		\noindent
		Let us introduce the abbreviation $M_N=\max_{0\le k<N^3}\alpha (y_k)$ for a moment. We can write
		\begin{align*}
		\P \left(Z_{0,N^3}^{P_1,\mathbf{y}}=Z_{0,N^3}^{P_2,\mathbf{y}},\,
		\sigma_{N} \le N^3
		\right) \le
		\P \left(
		U_{\sigma (j)+1}>1-\alpha (y_{\sigma (j)});\,
		j=1,\ldots,N-1
		\right) &\le \\
		\P \left(
		U_{\sigma (j)+1}>1-M_N;\,
		j=1,\ldots,N-1
		\right) &= \\
		\E \left[ 
		\P\left(
			U_{\sigma (N-1)+1}>1-M_N
		\middle|\G_{\sigma(N-1)}
		\right)
		\prod_{j=1}^{N-2}
		\ind_{\left\{
			U_{\sigma (j)+1}>1-M_N
			\right\}}
		\right].
		\end{align*}	
		Clearly, $U_{\sigma (N-1)+1}$ is independent of
		$\G_{\sigma (N-1)}$ so
		\begin{equation*}
		\P\left(
		U_{\sigma (N-1)+1}>1-M_N
		\middle|\G_{\sigma(N-1)}
		\right) =
		\P\left(
		U_{\sigma (N-1)+1}>1-M_N
		\right) = \max_{0\le k<N^3}\alpha (y_k).
		\end{equation*}
		Iteration of this argument leads to the following estimation.
		\begin{align*}
		\P \left(Z_{0,N^3}^{P_1,\mathbf{y}}=Z_{0,N^3}^{P_2,\mathbf{y}},\,
		\sigma_{N} \le N^3
		\right) \le \left(\max_{0\le k<N^3}\alpha (y_k)\right)^{N-1}
		\end{align*}
		By Lemma \ref{lem:sigma}, we have
		\begin{align*}
		&\P \left(Z_{0,N^3}^{P_1,\mathbf{y}}=Z_{0,N^3}^{P_2,\mathbf{y}}
		\right) \le 
		\P \left(Z_{0,N^3}^{P_1,\mathbf{y}}=Z_{0,N^3}^{P_2,\mathbf{y}},\,
		\sigma_{N} \le N^3
		\right)  +
		\P \left(\sigma_{N} > N^3\right) \\
		&\le \left(\max_{0\le k<N^3}\alpha (y_k)\right)^{N-1} +
			\bar{\ga}^{\left(\la-\frac{1}{2}\right)N^2}\frac{1-\sqrt{\bar{\ga}}}{2}
		\sum_{k=0}^{N-1}\E \left[V(Z_{0,kN^2+1}^{P_1,\mathbf{y}})+V(Z_{0,kN^2+1}^{P_2,\mathbf{y}})\right]
		\end{align*}
		which completes the proof.
	\end{proof}
	
	According to the following lemma, conditions on the maximum process of $\alpha (Y_k)$ appearing in the 
	statement of Theorem \ref{thm:LLN} can be translated to Assumption \ref{as:myas}, which 
	is somewhat more tractable.
	
	\begin{lemma}\label{lem:max}
		Let Assumption \ref{as:myas} be in force. Then, for every $N\in\N^+$ and $1\le p<\infty$,
		\begin{equation*}
			\sum_{N=1}^\infty \left\Vert
			\max_{0\le k <\lfrf{N^{1/M}}^M}
			\alpha (Y_k)^{\lfrf{N^{1/M}}-1}
			\right\Vert_p <\infty.
	 	\end{equation*}
		More precisely, there exists $c,\nu,\beta>0$ depending only on $M$, $p$ and $\theta$ such that
		\begin{equation*}
			\left\Vert
			\max_{0\le k <\lfrf{N^{1/M}}^M}
			\alpha (Y_k)^{\lfrf{N^{1/M}}-1}
			\right\Vert_p
			\le c \E\left[\alpha (Y_0)^{\lfrf{N^\beta}}\right]^{\frac{\nu}{\lfrf{N^\beta}^\theta}},\,N\in\N^+.
		\end{equation*}
	\end{lemma}
	\begin{proof}
		Let $\beta = \frac{1}{M(1-\theta)}$. For sufficiently large $N\in\N^+$,
		$\frac{1}{p}\frac{\lfrf{N^\beta}}{\lfrf{N^{1/M}}-1}>1$
		hence by Jensen's inequality and the strong stationarity of $Y_t$, $t\in\Z$, we have
		\begin{align*}
				\left\Vert
			\max_{0\le k <\lfrf{N^{1/M}}^M}
			\alpha (Y_k)^{\lfrf{N^{1/M}}-1}
			\right\Vert_p^{\frac{\lfrf{N^\beta}}{\lfrf{N^{1/M}}-1}}	
			\le
			\E\left[
			\max_{0\le k <\lfrf{N^{1/M}}^M}
			\alpha (Y_k)^{\lfrf{N^\beta}}
			\right] 
			\le 
			N \E\left[
			\alpha (Y_0)^{\lfrf{N^\beta}}
			\right]
		\end{align*}
		hence we obtain
		\begin{equation*}
			\left\Vert
			\max_{0\le k <\lfrf{N^{1/M}}^M}
			\alpha (Y_k)^{\lfrf{N^{1/M}}-1}
			\right\Vert_p
			\le 
			N^\frac{\lfrf{N^{1/M}}-1}{\lfrf{N^\beta}}
			\left(
			\E\left[
			\alpha (Y_0)^{\lfrf{N^\beta}}
			\right]^{\frac{1}{\lfrf{N^\beta}^\theta}}\right)^{\frac{\lfrf{N^{1/M}}-1}{\lfrf{N^\beta}^{1-\theta}}}
		\end{equation*}
		Taking into consideration that
		$\lim_{N\to\infty} \frac{\lfrf{N^{1/M}}-1}{\lfrf{N^\beta}}=0$
		and $\lim_{N\to\infty} \frac{\lfrf{N^{1/M}}-1}{\lfrf{N^\beta}^{1-\theta}} =1$,
		we conclude that there exists $c,\nu>0$ depending on $M$, $p$ and $\theta$ for which
		\begin{equation*}
		\left\Vert
		\max_{0\le k <\lfrf{N^{1/M}}^M}
		\alpha (Y_k)^{\lfrf{N^{1/M}}-1}
		\right\Vert_p
		\le c \E\left[\alpha (Y_0)^{\lfrf{N^\beta}}\right]^{\frac{\nu}{\lfrf{N^\beta}^\theta}},\,N\in\N^+
		\end{equation*}
		holds.
		
		By Assumption \ref{as:myas}, for some $0<\theta'<1$, $\E \left[\alpha (Y_0)^N\right]^\frac{1}{N^{\theta'}}\to 0$ as $N\to\infty$ and trivially, the
		same holds for any $0<\theta<\theta'$. Let us fix some $\theta\in (0,\theta')$. Then, by the previous point, we have the estimate
		\begin{equation*}
				\left\Vert
			\max_{0\le k <\lfrf{N^{1/M}}^M}
			\alpha (Y_k)^{\lfrf{N^{1/M}}-1}
			\right\Vert_p
			\le c \left(\E\left[\alpha (Y_0)^{\lfrf{N^\beta}}\right]^{\frac{\nu}{\lfrf{N^\beta}^{\theta'}}}\right)^{\lfrf{N^\beta}^{\theta'-\theta}},\,N\in\N^+,
		\end{equation*}
		where $c,\nu, \beta>0$ depends only on $M$, $p$ and $\theta$ thus for sufficiently
		large $n\in\N^+$,
			\begin{equation*}
		\sum_{N=n}^\infty \left\Vert
		\max_{0\le k <\lfrf{N^{1/M}}^M}
		\alpha (Y_k)^{\lfrf{N^{1/M}}-1}
		\right\Vert_p <c\sum_{N=n}^\infty \frac{1}{2^{\lfrf{N^\beta}^{\theta'-\theta}}}<\infty
		\end{equation*}
		which proves that
		\begin{equation*}
		\sum_{N=1}^\infty \left\Vert
		\max_{0\le k <\lfrf{N^{1/M}}^M}
		\alpha (Y_k)^{\lfrf{N^{1/M}}-1}
		\right\Vert_p <\infty.
		\end{equation*}
	\end{proof}

	\subsection{Pointwise convergence of kernels}\label{sec:proofs:TV}
	
	Let us introduce the sequence of probabilistic kernels 
	$\mu_k (\cdot, \cdot):\Y^\Z\times\B\to [0,1]$, $k\in\N$ such that for any fixed
	$\mathbf{y}\in\Y^\Z$,
	\begin{align}
	\begin{split}
	\mu_0 (\mathbf{y},\cdot) &= \delta_{x_0} \\
	\mu_n (\mathbf{y},\cdot) &= \law{Z_{0,n}^{x_0,S^{-n+1}\mathbf{y}}}.
	\end{split}
	\end{align}
	In this point, for typical $\mathbf{y}\in\Y^\Z$, we give an estimation for $\dtv (\mu_{n}(\mathbf{y},\cdot), \mu_{n+1}(\mathbf{y},\cdot))$, 
	moreover we prove that under Assumptions \ref{as:drift}, \ref{as:LT} and \ref{as:minor},
	for $\law{Y}-\Pas$ $\mathbf{y}\in\Y^\Z$, $\mu_{n}(\mathbf{y},\cdot)$, $n\in\N$
	converges to a probability measure in total variation distance.
	
	\begin{lemma}\label{lem:dtv}
		For $n\in\N^+$ and $1/2<\la<1$, we define the following sets.
		\begin{equation*}
			A_n^\la = \left\{
			\mathbf{y}\in\Y^\Z\middle|	\frac{\dtv (\mu_{n}(\mathbf{y},\cdot), \mu_{n+1}(\mathbf{y},\cdot))}{2} < \left(\max_{0\le k<\lfrf{n^{1/3}}^3} \alpha (y_k)\right)^{\lfrf{n^{1/3}}-1}
			+\bar{\ga}^{\left(\la-\frac{1}{2}\right)\lfrf{n^{1/3}}^2-n^{1/3}}
			\right\}
		\end{equation*}
		Then $\mathbf{Y}$ falls into $A_n^\la$ with large probability. More precisely, there exist $c,\nu>0$ for which
		\begin{equation*}
			\P \left(\mathbf{Y}\in A_n^\la\right)\ge 1-cn^{2/3} e^{-\nu n^{1/3}}.
		\end{equation*}
	\end{lemma}
	\begin{proof}
		
		According to the optimal transportation cost characterization of the total variation distance, we can write
		\begin{align*}
		\frac{1}{2}\dtv (\mu_{n}(\mathbf{y},\cdot), \mu_{n+1}(\mathbf{y},\cdot))
		&= \inf_{\ka\in\C (\mu_{n}(\mathbf{y},\cdot), \mu_{n+1}(\mathbf{y},\cdot))}
		\int_{\X\times\X} \ind_{x\ne y}\,\ka (\dint x, \dint y) \\
		&\le\P\left(
		Z_{0,n}^{x_0,S^{-n+1}\mathbf{y}}
		\ne 
		Z_{0,n+1}^{x_0,S^{-n}\mathbf{y}}
		\right) =
		\P\left(
		Z_{0,n}^{x_0,S^{-n+1}\mathbf{y}}
		\ne 
		Z_{0,n}^{T_0(y_{-n})x_0,S^{-n+1}\mathbf{y}} 
		\right) \\
		&\le 
		\P\left(
		Z_{0,\lfrf{n^{1/3}}^3}^{x_0,S^{-n+1}\mathbf{y}}
		\ne 
		Z_{0,\lfrf{n^{1/3}}^3}^{T_0(y_{-n})x_0,S^{-n+1}\mathbf{y}} 
		\right).
		\end{align*}
		
		If $\mathbf{y}\in\Y^\Z$ such that 
		$S^{-n+1}\mathbf{y}\in B_{n,3}^{\la}$, we can apply Lemma \ref{lem:central} since $x_0$ is deterministic and $T_0(y_{-n})x_0$
		is $\G_0$-measurable, and obtain
		\begin{align*}
			&\P\left(
			Z_{0,\lfrf{n^{1/3}}^3}^{x_0,S^{-n+1}\mathbf{y}}
			\ne 
			Z_{0,\lfrf{n^{1/3}}^3}^{T_0(y_{-n})x_0,S^{-n+1}\mathbf{y}} 
			\right)
			\le
			\left(\max_{0\le k<\lfrf{n^{1/3}}^3}\alpha (y_k)\right)^{\lfrf{n^{1/3}}-1} \\
			&+
			\bar{\ga}^{\left(\la-\frac{1}{2}\right) \lfrf{n^{1/3}}^2}\frac{1-\sqrt{\bar{\ga}}}{2}
			\sum_{k=0}^{\lfrf{n^{1/3}}-1}\E \left[V(Z_{0,k \lfrf{n^{1/3}}^2+1}^{x_0,S^{-n+1}\mathbf{y}})+V(Z_{0,k \lfrf{n^{1/3}}^2+1}^{T_0(y_{-n})x_0,S^{-n+1}\mathbf{y}})\right].
		\end{align*}
		By Lemma \ref{lem:ita}, Assumption \ref{as:drift} and the tower rule, for $0\le k<\lfrf{n^{1/3}}$, we have
		\begin{align*}
			\E \left[V(Z_{0,k \lfrf{n^{1/3}}^2+1}^{x_0,S^{-n+1}\mathbf{y}})+V(Z_{0,k \lfrf{n^{1/3}}^2+1}^{T_0(y_{-n})x_0,S^{-n+1}\mathbf{y}})\right] 
			= &\\
			\E \left(\left[Q(y_{k \lfrf{n^{1/3}}^2-n+1})\ldots Q(y_{-n+1})V\right](x_0)\right) 
			+ &\\
			\E \left(\left[Q(y_{k \lfrf{n^{1/3}}^2-n+1})\ldots Q(y_{-n+1})V\right](T_0(y_{-n})x_0)\right)
			\le&\\
			\left(V(x_0)+\E (V(T_0(y_{-n})x_0))\right)\,\prod_{r=-n+1}^{k \lfrf{n^{1/3}}^2-n+1}
			\ga (y_r)
			+ 2\sum_{r=-n+1}^{k \lfrf{n^{1/3}}^2-n+1}K(y_r)\prod_{j=r+1}^{k \lfrf{n^{1/3}}^2-n+1}
			\ga (y_j) \le&\\
			V(x_0)\left(1+\ga (y_{-n})\right)\,\prod_{r=-n+1}^{k \lfrf{n^{1/3}}^2-n+1}
			\ga (y_r)
			+ 2\sum_{r=-n}^{k \lfrf{n^{1/3}}^2-n+1}K(y_r)\prod_{j=r+1}^{k \lfrf{n^{1/3}}^2-n+1}
			\ga (y_j)
		\end{align*}
		where we have taken into account that $K(\cdot)\ge 1$.
		
		By the Markov inequality and the strong stationarity of $Y_t$, $t\in\Z$, we can write
		\begin{align*}
			\P\left(\mathbf{Y}\in\left\{\mathbf{y}\in\Y^\Z\middle| \sum_{k=0}^{\lfrf{n^{1/3}}-1}\E \left[V(Z_{0,k \lfrf{n^{1/3}}^2+1}^{x_0,S^{-n+1}\mathbf{y}})+V(Z_{0,k \lfrf{n^{1/3}}^2+1}^{T_0(y_{-n})x_0,S^{-n+1}\mathbf{y}})\right]
			\ge 
			\frac{2\bar{\ga}^{-n^{-1/3}}}{1-\sqrt{\bar{\ga}}}
			\right\}\right) \le&\\
			\bar{\ga}^{n^{1/3}}\frac{1-\sqrt{\bar{\ga}}}{2}
			\sum_{k=0}^{\lfrf{n^{1/3}}-1}
			\left\{
			V(x_0)\E\left[\left(1+\ga (Y_{-n})\right)\,\prod_{r=-n+1}^{k \lfrf{n^{1/3}}^2-n+1}
			\ga (Y_r)\right]
			\right.
			+&\\
			\left. 
			2\sum_{r=-n}^{k \lfrf{n^{1/3}}^2-n+1} 
			\E\left[K(Y_r)\prod_{j=r+1}^{k \lfrf{n^{1/3}}^2-n+1}
			\ga (Y_j)\right]
			\right\} 
			\le&\\
			c_1\bar{\ga}^{n^{1/3}}\frac{1-\sqrt{\bar{\ga}}}{2}
			\sum_{k=0}^{\lfrf{n^{1/3}}-1} 
			\left[
			V(x_0)\sqrt{\bar{\ga}}(1+\sqrt{\bar{\ga}})
			\bar{\ga}^{\frac{1}{2}k\lfrf{n^{1/3}}^2}
			+2\sum_{r=0}^{k\lfrf{n^{1/3}}^2+1}\bar{\ga}^{r/2}
			\right]
			\le&\\
			c_1\bar{\ga}^{n^{1/3}}\frac{1-\sqrt{\bar{\ga}}}{2}
			\left[V(x_0)
			\frac{\sqrt{\bar{\ga}}(1+\sqrt{\bar{\ga}})}{1-\bar{\ga}^{\frac{1}{2}\lfrf{n^{1/3}}^2}}
			+
			\frac{2\lfrf{n^{1/3}}}{1-\sqrt{\bar{\ga}}}
			\right]
			\le&\\
			c_1(V(x_0)+\lfrf{n^{1/3}})\bar{\ga}^{n^{1/3}},
		\end{align*}
		where $c_1$ is chosen such that $\E\left(K(Y_0)\prod_{t=1}^{n} \ga (Y_t)\right)\le c_1\bar{\ga}^{\frac{n}{2}}$
		, $n\in\N$ holds.
	
		So, by Lemma \ref{lem:cut} and by our previous considerations, there exist $c_1,c_2,\nu'>0$
		such that
		\begin{align*}
			\P \left(\mathbf{Y}\notin A_n^\la\right)
			&\le 
			\P \left(\mathbf{Y}\notin A_n^\la,\,S^{-n+1}\mathbf{Y}\in B_{n,3}^{\la}\right) 
			+
			\P \left(S^{-n+1}\mathbf{Y}\notin B_{n,3}^{\la}\right)
			\\
			&\le
			c_1(V(x_0)+\lfrf{n^{1/3}})\bar{\ga}^{n^{1/3}}
			+
			c_2 n^{2/3}e^{-\nu' n^{1/3}}.
		\end{align*}
		Clearly, there exists $c>0$ such that, for $\nu = \min (-\log\bar{\ga},\nu')$,
		\begin{equation*}
			c_1(V(x_0)+\lfrf{n^{1/3}})\bar{\ga}^{n^{1/3}}
			+
			c_2 n^{2/3}e^{-\nu' n^{1/3}} \le c n^{2/3}e^{-\nu n^{1/3}}
		\end{equation*}
		holds, which completes the proof.
	\end{proof}	
	
	The next lemma is a crucial ingredient of the proofs of both Theorem \ref{thm:TV} and \ref{thm:LLN}.
	
	\begin{lemma}\label{lem:dtvLp}
	Under Assumption \ref{as:myas}, there exists $1\le p<\infty$, such that
	\begin{equation*}
			\sum_{n=0}^\infty \left\Vert\dtv (\mu_{n}(\mathbf{Y},\cdot),\mu_{n+1}(\mathbf{Y},\cdot))\right\Vert_p <\infty.
	\end{equation*}		
	\end{lemma}
	\begin{proof}	
		According to Lemma \ref{lem:dtv}, there exist $c,\nu>0$ such that
		$\P \left(\mathbf{Y}\in A_n^\la\right)\ge 1-cn^{2/3}e^{-\nu n^{1/3}}$. So, we
		obtain the following upper bound for the general term
		\begin{align}\label{eq:dtvest}
		\begin{split}
		&\left\Vert\dtv (\mu_{n}(\mathbf{Y},\cdot),\mu_{n+1}(\mathbf{Y},\cdot))\right\Vert_p
		\le
		\left\Vert\dtv (\mu_{n}(\mathbf{Y},\cdot),\mu_{n+1}(\mathbf{Y},\cdot))\ind_{\mathbf{Y}\in A_n^\la}\right\Vert_p+ 2\,\P (\mathbf{Y}\notin A_n^\la) \\
		&\le
		2
		\left[
		\left\Vert\max_{0\le k<\lfrf{n^{1/3}}^3} \alpha (Y_k)^{\lfrf{n^{1/3}}-1}\right\Vert_p
		+\bar{\ga}^{\left(\la-\frac{1}{2}\right)\lfrf{n^{1/3}}^2-n^{1/3}}
		+cn^{2/3} e^{-\nu n^{1/3}}
		\right],\quad n\in\N^+
		\end{split}
		\end{align}
		which, by Lemma \ref{lem:max}, has a finite sum.
	\end{proof}
	
	\begin{corollary}\label{cor:dtv}
		For $\law{Y}-\Pas$ $\mathbf{y}$, $\mu_{n}(\mathbf{y},\cdot)$, $n\in\N$ is 
		convergent in the metric space $(\M_1, \dtv)$.{}
Let us denote this
		pointwise limit by $\mu_{\ast} (\mathbf{y},\cdot)$. $\mu_{\ast}$ is a probability kernel.
	\end{corollary}
	\begin{proof}
	We notice that the sequence of expected
		total variation distances has a finite sum, that is
		\begin{equation}
		\sum_{n=0}^\infty \E\left(\dtv (\mu_{n}(\mathbf{Y},\cdot),\mu_{n+1}(\mathbf{Y},\cdot))\right)<\infty
		\end{equation}
		which implies that $\mu_{n}(\mathbf{y},\cdot)$, $n\in\N$ is a Cauchy sequence ($P$-a.s.)
		hence it converges. It is not difficult
		to check that $\mu_{\ast}$ is indeed a probability kernel.
	\end{proof}

	According to the next remark, the $\law{Y}$-$\Pas$ existing pointwise limit 
	$\mu_{\ast}(\mathbf{y},\cdot)$ is independent of ${x_0}$ which will lead later to the 
	conclusion that the limit law $\mu_{\ast}$ in Theorem \ref{thm:TV} is independent of the initial value, as well.
	
	\begin{remark}\label{rem:icd}
		{\rm Signaling the dependence of $\mu_n (\mathbf{y},\cdot)$ on $x_0$, we write $\mu_n^{x_0}(\mathbf{y},\cdot)$. Applying the arguments presented in this section mutadis mutandis, 
		we obtain 
		\begin{align*}
	\frac{1}{2}\left\Vert\dtv \left(\mu_n^{x_0}(\mathbf{Y},\cdot),\mu_n^{x_0'}(\mathbf{Y},\cdot) \right)\right\Vert_p
			&\le
			\left\Vert\max_{0\le k<\lfrf{n^{1/3}}^3} \alpha (Y_k)^{\lfrf{n^{1/3}}-1}\right\Vert_p 
			\\
			&+\bar{\ga}^{\left(\la-\frac{1}{2}\right)\lfrf{n^{1/3}}^2-n^{1/3}}
			+c(x_0,x_0') n^{2/3} e^{-\nu n^{1/3}},\,\, n\in\N^+
		\end{align*}
	and thus by the triangle inequality, we get
	\begin{align*}
		\left\Vert\dtv \left(\mu_{\ast}^{x_0}(\mathbf{Y},\cdot),\mu_{\ast}^{x_0'}(\mathbf{Y},\cdot) \right)\right\Vert_p
		&\le
		\left\Vert\dtv \left(\mu_n^{x_0}(\mathbf{Y},\cdot),\mu_{\ast}^{x_0}(\mathbf{Y},\cdot) \right)\right\Vert_p
		+
		\left\Vert\dtv \left(\mu_{\ast}^{x_0'}(\mathbf{Y},\cdot),\mu_n^{x_0'}(\mathbf{Y},\cdot) \right)\right\Vert_p
		\\
		&+
		\left\Vert\dtv \left(\mu_n^{x_0}(\mathbf{Y},\cdot),\mu_n^{x_0'}(\mathbf{Y},\cdot) \right)\right\Vert_p,
	\end{align*}
	where the right-hand side tends to zero as $n\to\infty$ showing that
	for $\law{Y}-\Pas$ $\mathbf{y}\in\Y^\Z$, $\mu_{\ast}^{x_0}(\mathbf{y},\cdot) = \mu_{\ast}^{x_0'}(\mathbf{y},\cdot)$ holds.}	
	\end{remark}

	\subsection{Ergodicity of $\Phi (Z_t^\mathbf{y})$}\label{sec:proofs:Zt}
	
	Let $N\ge 1$ be an arbitrary natural number and $\mathbf{y}\in\Y^\Z$. Let us define the truncated process
	\begin{equation}\label{eq:Wproc}
	W_t(\mathbf{y}):=\left[
	\Phi \left(Z_{0,t}^{x_0,\mathbf{y}}\right)
	-
	\E \left(\Phi \left(Z_{0,t}^{x_0,\mathbf{y}}\right)\right)
	\right]\ind_{t\le\lfrf{N^{1/6}}^6},\, t\in\N. 
	\end{equation}

We will use the results of Section \ref{lm}.	
For $p\ge 1$, we introduce the quantities $M_p (W)=\sup_{t\in\N}\Vert W_t\Vert_p$ and
	\begin{equation*}
	\Gamma_p(W) =\sum_{\tau=1}^\infty \ga_p (W,\tau),
	\end{equation*} 
	where
	%\begin{equation*}
		$\ga_p (W,\tau) = \sup_{t\ge\tau}\Vert W_t-\E \left(W_t\mid \G_{t-\tau}^+\right)\Vert_p$, $\tau\ge 1$.
	%\end{equation*}
	If $\tau>\lfrf{N^{1/6}}^6$, then $\ga_p (W,\tau)=0$ thus $\Gamma_p(W)$ is finite which means 
	that $W_t$, $t\in\N$ is L-mixing of order $p$ with respect to $(\G_t,\G_t^+)$, $t\in\N$. 
	According to Lemma \ref{inek}, for $p\ge 2$, we have the estimate
	\begin{equation}\label{eq:mixing}
		\left\Vert \frac{1}{N}\sum_{t=1}^N W_t \right\Vert_p\le 
		C_p M_p^{1/2}(W)\sqrt{\frac{\Gamma_p (W)}{N}},
	\end{equation}
	where $C_p$ is a constant that depends neither on $N$ or $W$.

	Let us consider the estimate 
	\begin{equation*}
	\Gamma_p(W) =\sum_{\tau =1}^\infty \ga_p (W,\tau)\le 2\sqrt{N}\Vert \Phi\Vert_\infty+\sum_{\tau=\lfrf{N^{1/6}}^3+1}^{\lfrf{N^{1/6}}^6} \ga_p (W,\tau)
	\end{equation*}
	and for $s,t\in\N$, $t\ge s$ introduce the auxiliary process
	\begin{equation*}
	\widetilde{W}_{s,t}:=\left[
	\Phi \left(Z_{s,t}^{x_0,\mathbf{y}}\right)
	-
	\E \left(\Phi \left(Z_{0,t}^{x_0,\mathbf{y}}\right)\right)
	\right]\ind_{t\le\lfrf{N^{1/6}}^6}. 	
	\end{equation*}
	Note that, $W_{s,t}$ is measurable with respect to $\G_s^+$ moreover
	\begin{equation}\label{eq:diff}
	W_t-\widetilde{W}_{s,t} = \Phi \left(Z_{0,t}^{x_0,\mathbf{y}}\right)-\Phi \left(Z_{s,t}^{x_0,\mathbf{y}}\right)
	\end{equation}
	which will be important later.	
	
	For $\lfrf{N^{1/6}}^3<\tau\le \lfrf{N^{1/6}}^6$, there exists 
	$q,r\in\{0,1,\ldots,\lfrf{N^{1/6}}^3\}$ such that
	$\tau = q\lfrf{N^{1/6}}^3+r$, where $q\ge 1$. By our previous observation,
	$\widetilde{W}_{t-\lfrf{N^{1/6}}^3,t}$ is measurable with respect to $\G_{t-\lfrf{N^{1/6}}^3}^+$ moreover $\G_{t-\lfrf{N^{1/6}}^3}^+\subseteq \G_{t-\tau}^+$
	because $q\ge 1$ and thus $\widetilde{W}_{t-\lfrf{N^{1/6}}^3,t}$ is $\G_{t-\tau}^+$-measurable.
	
	By Lemma \ref{fyffes} and \eqref{eq:diff}, we can write
	\begin{align*}
		\ga_p (W,\tau) &= \max_{\tau\le t\le \lfrf{N^{1/6}}^6}
		\left\Vert 
		W_t-\E \left(
		W_t\mid \G_{t-\tau}^+
		\right)
		\right\Vert_p 
		\le
		\max_{\tau\le t\le \lfrf{N^{1/6}}^6}
		2\left\Vert 
		W_t-\widetilde{W}_{t-\lfrf{N^{1/6}}^3,t}
		\right\Vert_p \\
		&\le
		\max_{\lfrf{N^{1/6}}^3<t\le \lfrf{N^{1/6}}^6}
		2\left\Vert 
		W_t-\widetilde{W}_{t-\lfrf{N^{1/6}}^3,t}
		\right\Vert_p \\
		&=
		\max_{0<t\le \lfrf{N^{1/6}}^6-\lfrf{N^{1/6}}^3}
		2\left\Vert 
		W_{t+\lfrf{N^{1/6}}^3}-\widetilde{W}_{t,t+\lfrf{N^{1/6}}^3}
		\right\Vert_p \\
		&=
		\max_{0<t\le \lfrf{N^{1/6}}^6-\lfrf{N^{1/6}}^3}
		2\left\Vert 
		\Phi(Z_{0,t+\lfrf{N^{1/6}}^3}^{x_0,\mathbf{y}})-\Phi (Z_{t,t+\lfrf{N^{1/6}}^3}^{x_0,\mathbf{y}})
		\right\Vert_p \\
		&\le
		4\Vert\Phi\Vert_\infty
		\max_{0<t\le \lfrf{N^{1/6}}^6-\lfrf{N^{1/6}}^3}
		\P \left(
		Z_{0,t+\lfrf{N^{1/6}}^3}^{x_0,\mathbf{y}}
		\ne
		Z_{t,t+\lfrf{N^{1/6}}^3}^{x_0,\mathbf{y}}
		\right).
	\end{align*}
	
	We substitute this back into \eqref{eq:mixing} and we obtain
	the following upper bound 
	\begin{equation}\label{eq:bound}
		\left\Vert \frac{1}{N}\sum_{t=1}^N W_t \right\Vert_p\le 
		2 C_p \Vert\Phi\Vert_\infty \left(
		\frac{1}{\sqrt{N}}+
		2\max_{0<t\le \lfrf{N^{1/6}}^6-\lfrf{N^{1/6}}^3}
		\P \left(
		Z_{0,t+\lfrf{N^{1/6}}^3}^{x_0,\mathbf{y}}
		\ne
		Z_{t,t+\lfrf{N^{1/6}}^3}^{x_0,\mathbf{y}}
		\right)
		\right)^{1/2}.
	\end{equation}
	
	The content of the next lemma is that there exist ``large'' sets of the environment for which good enough couplings 
	occur. 
	\begin{lemma}\label{lem:WtypicalY}
	For $N\in\N^+$, $1/2<\la<1$ and $0<t\le \lfrf{N^{1/6}}^6-\lfrf{N^{1/6}}^3$, let us define the following sets.	
	\begin{equation*}
	C_{N,t}^\la = \left\{
	\mathbf{y}\in\Y^\Z\middle|	
	\P \left(
	Z_{0,\lfrf{N^{1/6}}^3}^{Z''_t,S^t\mathbf{y}}
	\ne
	Z_{0,\lfrf{N^{1/6}}^3}^{x_0,S^t\mathbf{y}}
	\right)
	<
	\left(\max_{0\le k<\lfrf{N^{1/6}}^3}\alpha (y_{k+t})\right)^{\lfrf{N^{1/6}}-1} +
	\bar{\ga}^{\left(\la-\frac{1}{2}\right)\lfrf{N^{1/6}}^2-N^{1/6}}
	\right\}
	\end{equation*} 	
	Then there exist $c,\nu>0$ such that
	\begin{equation*}
		\P \left( \mathbf{Y}\in C_{N}^\la\right)\ge 1-c N^{7/6} e^{-\nu N^{1/6}},
		\text{ where }
		C_{N}^\la=\bigcap_{t=1}^{\lfrf{N^{1/6}}^6-\lfrf{N^{1/6}}^3} C_{N,t}^\la.
	\end{equation*}	
	\end{lemma}
	\begin{proof}
		Let $N\in\N^+$ and $0<t\le \lfrf{N^{1/6}}^6-\lfrf{N^{1/6}}^3$ be arbitrary and fixed.
		We have the following identities
		\begin{align*}
		\P \left(
		Z_{0,t+\lfrf{N^{1/6}}^3}^{x_0,\mathbf{y}}
		\ne
		Z_{t,t+\lfrf{N^{1/6}}^3}^{x_0,\mathbf{y}}
		\right) 
		=
		\P \left(
		Z_{t,t+\lfrf{N^{1/6}}^3}^{Z',\mathbf{y}}
		\ne
		Z_{t,t+\lfrf{N^{1/6}}^3}^{x_0,\mathbf{y}}
		\right) 
		=
		\P \left(
		Z_{0,\lfrf{N^{1/6}}^3}^{Z''_t,S^t\mathbf{y}}
		\ne
		Z_{0,\lfrf{N^{1/6}}^3}^{x_0,S^t\mathbf{y}}
		\right),
		\end{align*}
		where $Z'=Z_{0,t}^{x_0,\mathbf{y}}$ and $Z''_t=Z_{-t,0}^{x_0,S^t\mathbf{y}}$.
		
		If $\mathbf{y}\in\Y^\Z$ such that $S^t\mathbf{y}\in B_{\lfrf{N^{1/6}}^6,6}^\la$, 
		$t=0,1,\ldots,\lfrf{N^{1/6}}-1$, then by Remark \ref{rem:six},
		$S^t\mathbf{y}\in B_{\lfrf{N^{1/6}}^3,3}^\la$, $0<t\le\lfrf{N^{1/6}}^6-\lfrf{N^{1/6}}^3$, furthermore
		$Z''_t$ is $\G_0$-measurable hence we can apply Lemma \ref{lem:central} thus
		we obtain
		\begin{align*}
		&\P \left(
		Z_{0,\lfrf{N^{1/6}}^3}^{Z''_t,S^t\mathbf{y}}
		\ne
		Z_{0,\lfrf{N^{1/6}}^3}^{x_0,S^t\mathbf{y}}
		\right)
		\le 
		\left(\max_{0\le k<\lfrf{N^{1/6}}^3}\alpha (y_{k+t})\right)^{\lfrf{N^{1/6}}-1} \\
		&+
		\bar{\ga}^{\left(\la-\frac{1}{2}\right)\lfrf{N^{1/6}}^2}\frac{1-\sqrt{\bar{\ga}}}{2}
		\sum_{k=0}^{\lfrf{N^{1/6}}-1}
		\E\left[V(Z_{0,k\lfrf{N^{1/6}}^2+1}^{Z''_t,S^t\mathbf{y}})+V(Z_{0,k\lfrf{N^{1/6}}^2+1}^{x_0,S^t\mathbf{y}})\right].
		\end{align*}
		
			By Lemma \ref{lem:ita}, Assumption \ref{as:drift} and the tower rule, for $0\le k<\lfrf{N^{1/6}}$, we have
		\begin{align*}
		\E\left[V(Z_{0,k\lfrf{N^{1/6}}^2+1}^{Z''_t,S^t\mathbf{y}})+V(Z_{0,k\lfrf{N^{1/6}}^2+1}^{x_0,S^t\mathbf{y}})\right]
		= &\\
			\E \left(\left[Q(y_{k\lfrf{N^{1/6}}^2+t})\ldots Q(y_{t})V\right](Z_{-t,0}^{x_0,S^t\mathbf{y}})\right) 
		+
		\E \left(\left[Q(y_{k\lfrf{N^{1/6}}^2+t})\ldots Q(y_{t})V\right](x_0)\right)
		=&\\
		\E \left(\left[Q(y_{k\lfrf{N^{1/6}}^2+t})\ldots Q(y_{0})V\right](x_0)\right) 
		+
		\E \left(\left[Q(y_{k\lfrf{N^{1/6}}^2+t})\ldots Q(y_{t})V\right](x_0)\right)		
		\le&\\
		V(x_0)\left[ 
		\prod_{r=0}^{k\lfrf{N^{1/6}}^2+t}
		\ga (y_r)
		+
		\prod_{r=t}^{k\lfrf{N^{1/6}}^2+t}
		\ga (y_r)
		\right]
		+&\\
		\sum_{r=t}^{k\lfrf{N^{1/6}}^2+t}K(y_r)\prod_{j=r+1}^{k\lfrf{N^{1/6}}^2+t}
		\ga (y_j) 
		+
		\sum_{r=0}^{k\lfrf{N^{1/6}}^2+t}K(y_r)\prod_{j=r+1}^{k\lfrf{N^{1/6}}^2+t}
		\ga (y_j) 
		\le&\\
		V(x_0)\left[ 
		\prod_{r=0}^{k\lfrf{N^{1/6}}^2+t}
		\ga (y_r)
		+
		\prod_{r=t}^{k\lfrf{N^{1/6}}^2+t}
		\ga (y_r)
		\right]
		+
		2
		\sum_{r=0}^{k\lfrf{N^{1/6}}^2+t}K(y_r)\prod_{j=r+1}^{k\lfrf{N^{1/6}}^2+t}
		\ga (y_j).
		\end{align*}
		
	By the Markov inequality and the strong stationarity of $Y_t$, $t\in\Z$, we can write
	\begin{align*}
	\P\left(\mathbf{Y}\in\left\{\mathbf{y}\in\Y^\Z\middle| \sum_{k=0}^{\lfrf{N^{1/6}}-1}
	\E\left[V(Z_{0,k\lfrf{N^{1/6}}^2+1}^{Z''_t,S^t\mathbf{y}})+V(Z_{0,k\lfrf{N^{1/6}}^2+1}^{x_0,S^t\mathbf{y}})\right]
	\ge 
	\frac{2\bar{\ga}^{-N^{-1/6}}}{1-\sqrt{\bar{\ga}}}
	\right\}\right) 
	\le&\\
	\bar{\ga}^{N^{1/6}}\frac{1-\sqrt{\bar{\ga}}}{2}
	\sum_{k=0}^{\lfrf{N^{1/6}}-1}
	\left\{
	V(x_0)\E\left[ 
	\prod_{r=0}^{k\lfrf{N^{1/6}}^2+t}
	\ga (Y_r)
	+
	\prod_{r=t}^{k\lfrf{N^{1/6}}^2+t}
	\ga (Y_r)
	\right]
	\right.
	+&\\
	\left. 2\sum_{r=0}^{k\lfrf{N^{1/6}}^2+t}
	\E\left[
	K(Y_r)\prod_{j=r+1}^{k\lfrf{N^{1/6}}^2+t}
	\ga (Y_j)
	\right]
	\right\} 
	\le&\\
	c_1\bar{\ga}^{N^{1/6}}\frac{1-\sqrt{\bar{\ga}}}{2}
	\sum_{k=0}^{\lfrf{N^{1/6}}-1}
	\left[
	V(x_0)\sqrt{\bar{\ga}}(1+\bar{\ga}^{t/2})
	\bar{\ga}^{\frac{1}{2}k\lfrf{N^{1/6}}^2}
	+2\sum_{r=0}^{k\lfrf{N^{1/6}}^2+t}\bar{\ga}^{r/2}
	\right]
	\le&\\
	c_1\bar{\ga}^{N^{1/6}}\frac{1-\sqrt{\bar{\ga}}}{2}
	\left[
	V(x_0)\frac{\sqrt{\bar{\ga}}(1+\bar{\ga}^{t/2})}{1-\bar{\ga}^{\frac{1}{2}\lfrf{N^{1/6}}^2}}
		+
	\frac{2\lfrf{N^{1/6}}}{1-\sqrt{\bar{\ga}}}
	\right]\le&\\
	c_1(V(x_0)+\lfrf{N^{1/6}})\bar{\ga}^{N^{1/6}},
	\end{align*}
	where $c_1$ is chosen such that $\E\left(K(Y_0)\prod_{t=1}^{n} \ga (Y_t)\right)\le c_1\bar{\ga}^{\frac{n}{2}}$, $n\in\N$ holds.
	
	So, by Lemma \ref{lem:cut} and our previous considerations, there exists $c_1,c_2,\nu'>0$
	such that
	\begin{align*}
	\P \left(\bigcup_{t=1}^{\lfrf{N^{1/6}}^6-\lfrf{N^{1/6}}^3}\mathbf{Y}\notin C_{n,t}^\la\right)
	&\le
	\sum_{t=1}^{\lfrf{N^{1/6}}^6-\lfrf{N^{1/6}}^3} 
	\P \left(\mathbf{Y}\notin C_{n,t}^\la,\,
	\bigcap_{s=r}^{\lfrf{N^{1/6}}-1}S^r\mathbf{y}\in B_{\lfrf{N^{1/6}}^6,6}^\la
	\right) 
	\\
	&+
	\P \left(\bigcup_{r=0}^{\lfrf{N^{1/6}}-1}S^r\mathbf{y}\notin B_{\lfrf{N^{1/6}}^6,6}^\la\right)
	\\
	&\le
	c_1N(V(x_0)+\lfrf{N^{1/6}})\bar{\ga}^{N^{1/6}}
	+
	c_2 N e^{-\nu' N^{1/6}}.
	\end{align*}
	Clearly, there exists $c>0$ such that, for which $\nu = \min (-\log\bar{\ga},\nu')$,
	\begin{equation*}
	c_1N(V(x_0)+\lfrf{N^{1/6}})\bar{\ga}^{N^{1/6}}
	+
	c_2 N e^{-\nu' N^{1/6}} \le c N^{7/6} e^{-\nu N^{1/6}}
	\end{equation*}
	holds which completes the proof.
	\end{proof}

	Finally, we arrive at the following important result which will play a central role in the proof of Theorem \ref{thm:LLN}.
	Combining our estimates so far, we can bound the $L_{p}$-norm of the functional averages.
	
	\begin{lemma}\label{lem:central2}
		There exists $\tilde{c}(p,\bar{\ga},\la)>0$ depending only on
		$p$, $\bar{\ga}$ and $\la$ such that
		\begin{equation*}
		\E^{1/p}\left[
		\left\Vert 
		\frac{1}{N}\sum_{t=1}^{N} W_t(\mathbf{Y})
		\right\Vert_p^p
		\right] 
		\le\tilde{c}(p,\bar{\ga},\la)
		\Vert\Phi\Vert_\infty
		\left(
		N^{-1/4}+
		\left\Vert
		\max_{0\le k<\lfrf{N^{1/6}}^6}\alpha (Y_{k})^{\lfrf{N^{1/6}}-1}
		\right\Vert_{p/2}^{1/2}
		\right).
		\end{equation*}
		
	\end{lemma}
	\begin{proof}
		Without the loss of generality, we may assume that $p\ge 2$. Clearly on
		$C_N^\la$
		\begin{align*}
			\max_{0<t\le \lfrf{N^{1/6}}^6-\lfrf{N^{1/6}}^3}
			\P \left(
			Z_{0,t+\lfrf{N^{1/6}}^3}^{x_0,\mathbf{y}}
			\ne
			Z_{t,t+\lfrf{N^{1/6}}^3}^{x_0,\mathbf{y}}
			\right) \le&\\ 
			\max_{0<t\le \lfrf{N^{1/6}}^6-\lfrf{N^{1/6}}^3}
			\left(\max_{0\le k<\lfrf{N^{1/6}}^3}\alpha (y_{k+t})\right)^{\lfrf{N^{1/6}}-1} +
			\bar{\ga}^{\left(\la-\frac{1}{2}\right)\lfrf{N^{1/6}}^2-N^{1/6}} \le&\\
			\left(\max_{0\le k<\lfrf{N^{1/6}}^6}\alpha (y_{k})\right)^{\lfrf{N^{1/6}}-1} +
			\bar{\ga}^{\left(\la-\frac{1}{2}\right)\lfrf{N^{1/6}}^2-N^{1/6}} &
		\end{align*}
		holds, hence by \eqref{eq:bound},
		we can write
		\begin{align*}
			\E^{1/p}\left[
			\left\Vert 
			\frac{1}{N}\sum_{t=1}^{N} W_t(\mathbf{Y})
			\right\Vert_p^p
			\right] 
			\le&\\
			\E^{1/p}\left[
			\left\Vert 
			\frac{1}{N}\sum_{t=1}^{N} W_t(\mathbf{Y})
			\right\Vert_p^p\ind_{\mathbf{Y}\in C_N^\la}
			\right]
			+
			\E^{1/p}\left[
			\left\Vert 
			\frac{1}{N}\sum_{t=1}^{N} W_t(\mathbf{Y})
			\right\Vert_p^p\ind_{\mathbf{Y}\notin C_N^\la}
			\right]
			\le&\\
			2\Vert\Phi\Vert_\infty \left[C_p
			\left(
			\frac{1}{\sqrt{N}}+2\bar{\ga}^{\left(\la-\frac{1}{2}\right)\lfrf{N^{1/6}}^2-N^{1/6}}
			+2
			\left\Vert
			\max_{0\le k<\lfrf{N^{1/6}}^6}\alpha (Y_{k})^{\lfrf{N^{1/6}}-1}
			\right\Vert_{p/2} 
			\right)^{1/2}
			\right.+&\\
			\left.		
			+\P\left(\mathbf{Y}\notin C_N^\la\right)^{1/p}
			\right]&\\
		\end{align*}
		The square root function is subadditive hence by Lemma \ref{lem:WtypicalY}, there exists $\tilde{c}(p,\bar{\ga},\la)$ depending only on $p$, $\bar{\ga}$ and $\la$ such that
		\begin{align*}
			C_p
			\left(
				\frac{1}{\sqrt{N}}+2\bar{\ga}^{\left(\la-\frac{1}{2}\right)\lfrf{N^{1/6}}^2-N^{1/6}}
				+2
				\left\Vert
				\max_{0\le k<\lfrf{N^{1/6}}^6}\alpha (Y_{k})^{\lfrf{N^{1/6}}-1}
				\right\Vert_{p/2} 
			\right)^{1/2}	
			+\P\left(\mathbf{Y}\notin C_N^\la\right)^{1/p} \le&\\
			C_p\left(
			N^{-1/4}+
			\sqrt{2}\bar{\ga}^{\frac{1}{2}\left[\left(\la-\frac{1}{2}\right)\lfrf{N^{1/6}}^2-N^{1/6}\right]}
			+
			\sqrt{2}
				\left\Vert
			\max_{0\le k<\lfrf{N^{1/6}}^6}\alpha (Y_{k})^{\lfrf{N^{1/6}}-1}
			\right\Vert_{p/2}^{1/2}
			\right) 
			+&\\
			c^{1/p}N^{\frac{7}{6p}}e^{-\frac{\nu}{p}N^{1/6}}
			\le&\\
			\frac{\tilde{c}(p,\bar{\ga},\la)}{2}
			\left(
			N^{-1/4}+
				\left\Vert
			\max_{0\le k<\lfrf{N^{1/6}}^6}\alpha (Y_{k})^{\lfrf{N^{1/6}}-1}
			\right\Vert_{p/2}^{1/2}
			\right)
		\end{align*}
		Finally, we obtain the desired upper bound
		\begin{equation*}
		\E^{1/p}\left[
		\left\Vert 
		\frac{1}{N}\sum_{t=1}^{N} W_t(\mathbf{Y})
		\right\Vert_p^p
		\right] 
		\le 	
		\tilde{c}(p,\bar{\ga},\la)
		\Vert\Phi\Vert_\infty
		\left(
		N^{-1/4}+
		\left\Vert
		\max_{0\le k<\lfrf{N^{1/6}}^6}\alpha (Y_{k})^{\lfrf{N^{1/6}}-1}
		\right\Vert_{p/2}^{1/2}
		\right)
		\end{equation*}
		which completes the proof.
	\end{proof}
	
	\subsection{Proof of Theorem \ref{thm:TV}}
	
	For $A\in\B$ arbitrary and any decomposition of $A$ into disjoint and measurable sets $A=\cup_i A_i$,
	we have
	\begin{align*}
		\sum_i |\E(\mu_{n}(\mathbf{Y},A_i))-\E(\mu_{n+1}(\mathbf{Y},A_i))|
		&\le
		\E\left[\sum_i |\mu_{n}(\mathbf{Y},A_i)-\mu_{n+1}(\mathbf{Y},A_i)|\right]
		\\
		&\le 
		\E \left(\dtv (\mu_{n}(\mathbf{Y},\cdot),\mu_{n+1}(\mathbf{Y},\cdot))\right).
	\end{align*}
	Taking the supremum of the left-hand side, we get 
	\begin{equation*}
	\dtv\left(\E(\mu_{n}(\mathbf{Y},\cdot)),\E(\mu_{n+1}(\mathbf{Y},\cdot))\right)\le\E \left(\dtv (\mu_{n}(\mathbf{Y},\cdot),\mu_{n+1}(\mathbf{Y},\cdot))\right).
	\end{equation*}
	As easily seen, $\mu_{n}(A)=\E (\mu_{n}(\mathbf{Y},A))$ holds for $A\in\B$, so we infer that
	\begin{equation*}
		\dtv (\mu_{n}, \mu_{n+1}) \le \E \left(\dtv (\mu_{n}(\mathbf{Y},\cdot),\mu_{n+1}(\mathbf{Y},\cdot))\right).
	\end{equation*}
	Then it follows from Corollary \ref{cor:dtv} that
	\begin{equation*}
		\sum_{n=1}^\infty \dtv (\mu_n, \mu_{n+1}) < \infty,
	\end{equation*}
	hence $\mu_n$, $n\in\N$ is a Cauchy sequence in the complete metric space $(\M_1, \dtv)$.
	Hence it converges to some probability $\mu_{\ast\ast}$ as $n\to\infty$.
	The claimed convergence rate follows from \eqref{eq:dtvest} with $p=1$:
	\begin{align*}
		\dtv (\mu_{N}, \mu_{\ast\ast}) 
		&\le 
		\sum_{n=N}^{\infty} \dtv (\mu_{n}, \mu_{n+1}) \le
		\sum_{n=N}^{\infty} \E \left(\dtv (\mu_{n}(\mathbf{Y},\cdot),\mu_{n+1}(\mathbf{Y},\cdot))\right) \\
		&\le 
		2\sum_{n=N}^{\infty}
		\left[
		\E\left(\max_{0\le k<\lfrf{n^{1/3}}^3} \alpha (Y_k)^{\lfrf{n^{1/3}}-1}\right)
		+\bar{\ga}^{\left(\la-\frac{1}{2}\right)\lfrf{n^{1/3}}^2-n^{1/3}}
		+cn^{2/3} e^{-\nu n^{1/3}}
		\right].
	\end{align*}
	
	It remains to prove that $\mu_{\ast}$ and $\mu_{\ast\ast}$ coincide. It is clear that
	for every $A\in\B$,
	\begin{align*}
		\mu_{\ast\ast} (A) = \lim_{n\to\infty} \mu_{n} (A) =
		\lim_{n\to\infty} \E (\mu_{n} (\mathbf{Y}, A)) =
		\E (\lim_{n\to\infty} \mu_{n} (\mathbf{Y}, A)) =
		\E (\mu_{\ast} (\mathbf{Y}, A)) =
		\mu_{\ast} (A),
	\end{align*}
	hence $\mu_{\ast} = \mu_{\ast\ast}$. By Remark \ref{rem:icd}, easily follows that the limit distribution $\mu_{\ast}$ is independent of the initial state $x_0$.
	\qed

	\subsection{Proof of Theorem \ref{thm:LLN}}
	
	Let $N\ge 1$ arbitrary integer, $1\le p<\infty$ and consider the following estimate.
	\begin{align}\label{eq:triangle}
	\begin{split}
	&
	\left\Vert
	\frac{1}{N}\sum_{t=1}^N \Phi \left(Z_{0,t}^{x_0,\mathbf{Y}}\right) - \int_\X \Phi (z) \,\mu_\ast (\dint z)
	\right\Vert_p 
	\le
	\left\Vert
	\frac{1}{N}\sum_{t=0}^{N-1}\int_\X \Phi (z)\,\left[\mu_\ast (S^{t}\mathbf{Y},\dint z)
	-\mu_{\ast} (\dint z)\right]
	\right\Vert_p \\
	&+
	\left\Vert
	\frac{1}{N}\sum_{t=1}^{N} \int_\X \Phi(z)\,(\mu_t-\mu_\ast) (S^{t-1}\mathbf{Y},\dint z)
	\right\Vert_p \\
	&+\left\Vert
	\frac{1}{N}\sum_{t=1}^N \left(\Phi \left(Z_{0,t}^{x_0,\mathbf{Y}}\right) - \int_\X \Phi (z)\,
	\mu_t (S^{t-1}\mathbf{Y},\dint z) \right)
	\right\Vert_p
	\end{split} 	
	\end{align}
	
	As we already mentioned in Section \ref{sec:main}, the stochastic process $Y$ is strongly stationary and ergodic hence the left shift
	$S:\Y^\Z\to\Y^\Z$ is an ergodic endomorphism of the probability space $(\Y^\Z,\A^{\otimes\Z},\law{Y})$, moreover 
	$\Y^\Z\ni\mathbf{y}\mapsto\int_\X \Phi (z)\,\mu_\ast (\mathbf{y},\dint z)$
	is obviously in $L^1$ hence Birkhoff's ergodic theorem implies that
	\begin{align*}
	\frac{1}{N}\sum_{t=0}^{N-1}\int_\X \Phi (z)\,\mu_\ast (S^{t}\mathbf{Y},\dint z)
	\to
	\int_\X \Phi (z) \,\mu_{\ast} (\dint z),\,N\to\infty,
	\end{align*}
	almost surely and also in $L^p$ due to Lebesgue's dominated convergence theorem.
	
	By the strong stationary property of $Y$ again, for the second term, we have
	\begin{align*}
		\left\Vert
		\frac{1}{N}\sum_{t=1}^{N} \int_\X \Phi(z)\,(\mu_t-\mu_\ast) (S^{t-1}\mathbf{Y},\dint z)
		\right\Vert_p 
		&\le 
		\frac{\Vert\Phi\Vert_\infty}{N}\sum_{t=1}^N 
		\Vert 
		\dtv \left(\mu_t(\mathbf{Y},\cdot),\mu_{\ast}(\mathbf{Y},\cdot)\right)
		\Vert_p \\
		&\le
		\frac{\Vert\Phi\Vert_\infty}{N}\sum_{t=1}^N
		\sum_{n=t}^{\infty} 
		\Vert 
		\dtv \left(\mu_n(\mathbf{Y},\cdot),\mu_{n+1}(\mathbf{Y},\cdot)\right)
		\Vert_p
	\end{align*}
	which is a C\`{e}saro sum and due to Lemma \ref{lem:dtvLp}, the general
	term tends to zero thus we obtain
	\begin{equation*}
		\left\Vert
		\frac{1}{N}\sum_{t=1}^{N} \int_\X \Phi(z)\,(\mu_t-\mu_\ast) (S^{t-1}\mathbf{Y},\dint z)
		\right\Vert_p \to 0,\,N\to\infty.
	\end{equation*}
	
	Finally, due to the definition of $\mu_t (\cdot,\cdot)$, for any fixed $\mathbf{y}\in\Y^\Z$, the law of 
	$Z_{0,t}^{x_0,\mathbf{y}}$ equals to $\mu_t (S^{t-1}\mathbf{y},\cdot)$ hence for the last term, we have
	\begin{align*}
		&\left\Vert
		\frac{1}{N}\sum_{t=1}^N \left(\Phi \left(Z_{0,t}^{x_0,\mathbf{Y}}\right) - \int_\X \Phi (z)\,
		\mu_t (S^{t-1}\mathbf{Y},\dint z) \right)
		\right\Vert_p \le 
		2\frac{N-\lfrf{N^{1/6}}^6}{N}\Vert\Phi\Vert_\infty + \\
		&+
		\E^{1/p}\left[
		\E\left(\left|
		\frac{1}{N}\sum_{t=1}^{\lfrf{N^{1/6}}^6} \left(\Phi \left(Z_{0,t}^{x_0,\mathbf{Y}}\right) - \int_\X \Phi (z)\,
		\mu_t (S^{t-1}\mathbf{Y},\dint z) \right)
		\right|^p 
		\middle|\sigma (\mathbf{Y})\right)
		\right] \\
		&\le
		\frac{12\Vert\Phi\Vert_\infty}{N^{1/6}}
		+
		\E^{1/p}\left[
		\left\Vert 
		\frac{1}{N}\sum_{t=1}^{N} W_t(\mathbf{Y})
		\right\Vert_p^p
		\right].
	\end{align*}
	
	According to Lemma \ref{lem:central2}, exists $\tilde{c}(p,\bar{\ga},\la)>0$
	such that
	\begin{equation*}
	\E^{1/p}\left[
	\left\Vert 
	\frac{1}{N}\sum_{t=1}^{N} W_t(\mathbf{Y})
	\right\Vert_p^p
	\right] 
	\le\tilde{c}(p,\bar{\ga},\la)
	\Vert\Phi\Vert_\infty
	\left(
	N^{-1/4}+
	\left\Vert
	\max_{0\le k<\lfrf{N^{1/6}}^6}\alpha (Y_{k})^{\lfrf{N^{1/6}}-1}
	\right\Vert_{p/2}^{1/2}
	\right)
	\end{equation*}
	hence we obtain
	\begin{align*}
	\left\Vert
	\frac{1}{N}\sum_{t=1}^N \left(\Phi \left(Z_{0,t}^{x_0,\mathbf{Y}}\right) - \int_\X \Phi (z)\,
	\mu_t (S^{t-1}\mathbf{Y},\dint z) \right)
	\right\Vert_p &\le\\
	\frac{12\Vert\Phi\Vert_\infty}{N^{1/6}}
	+
	\tilde{c}(p,\bar{\ga},\la)
	\Vert\Phi\Vert_\infty
	\left(
	N^{-1/4}+
	\left\Vert
	\max_{0\le k<\lfrf{N^{1/6}}^6}\alpha (Y_{k})^{\lfrf{N^{1/6}}-1}
	\right\Vert_{p/2}^{1/2}
	\right),&
	\end{align*}
	where by Lemma \ref{lem:max}, the upper bound tends to zero as $N\to\infty$.

	To sum up,
	\begin{equation*}
		\left\Vert
		\frac{1}{N}\sum_{t=1}^N \Phi (X_t) - \int_\X \Phi (z)\,\mu_\ast (\dint z)
		\right\Vert_p \to 0,\,N\to\infty
	\end{equation*}
	because the laws of $X_t$ and $Z_{0,t}^{x_0,\mathbf{Y}}$ coincides. This completes the proof of Theorem \ref{thm:LLN}.
	
	\begin{remark}\label{borkoff}
		{\rm Birkhoff's ergodic theorem does not provide an uppper bound for the difference
		between time and space averages hence, we have a convergence rate for every term
		in \eqref{eq:triangle} except for the first one. However, in the ideal case this term 
		is of the order $1/\sqrt{N}$ and this can be shown for $Y$ with suitably
		favourable ergodic properties.}
	\end{remark}

	\subsection{Proof of Theorem \ref{thm:gen}}\label{sec:proof:gen}
	
	In order to keep the explanation compact, we omit the details of the proof and just 
	indicate how our lemmas should be modified to make them applicable in this setting. 
	
	Let $p\ge 1$ be as in Assumption \ref{as:drift1}, \ref{as:LT1} and \ref{as:minor1} and 
	we introduce the auxiliary processes 
	\begin{equation*}
	\left(\tilde{X}_t^j\right)_{t\in\N}=
	\left(Z_{0,pt}^{Z'_j,\mathbf{Y}}\right)_{t\in\N},\,\,
	j=0,\ldots,p-1,
	\end{equation*}
	where $Z'_j=Z_{-j,0}^{x_0,\mathbf{Y}}$.	
	Note that $\tilde{X}_t^j=Z_{-j,pt}^{x_0,\mathbf{Y}}=Z_{0,pt+j}^{x_0,S^{-j}\mathbf{Y}}$ and hence, by the 
	strong stationarity of $Y_t$, $t\in\Z$, the laws of $\tilde{X}_t^j$ and $X_{pt+j}$ coincide. 
	Furthermore, the process $\tilde{X}_t^j$, $t\in\N$ also can be considered
	as Markov chain in a random environment driven by
	\begin{equation*}
	\tilde{Y}_t^j=(Y_{p(t-1)+j},\ldots, Y_{pt+j-1})\in\Y^p,\,\, t\in\Z.
	\end{equation*}
	We denote the parametric family of stochastic kernels corresponding to $\tilde{X}_t^j$ by $\tilde{Q}$.
	It is easy to check that $\tilde{Q}:\Y^p\times\X\times\B\to [0,1]$ is independent of 
	$j$ and arises as the $p$-times composition of $Q(y)s$ with itself, that is,
	\begin{equation*}
	[\tilde{Q}((y_{1},\ldots, y_{p}))\phi](x)=[Q(y_p)\ldots Q(y_1)\phi](x)
	\end{equation*}
	holds for arbitrary non-negative measurable $\phi:\X\to\R_+$, $(y_{1},\ldots, y_{p})\in\Y^p$ and $x\in\X$. 
	
	Assumptions in 
	Section \ref{sec:ram} imply that Assumptions \ref{as:drift}, \ref{as:LT} and \ref{as:minor} 
	hold for the processes $\left(\tilde{X}_t^j\right)_{t\in\N}$, 
	$j\in\{0,\ldots,p-1\}$, where the corresponding $\alpha$, $\kappa$ and $K$ are independent of the environment.
	In Assumption \ref{as:drift1}, the one-step bound i.e.\ inequality \eqref{eq:extra} gives the estimate
	\begin{align}\label{eq:univ}
	\E \left(V(Z_{-j,0}^{x_0,\mathbf{y}})\right) 
	\le
	K^j V(x_0) + \sum_{i=1}^j K^i\le K^p (V(x_0)+p)		
	\end{align}
	which is uniform in $\mathbf{y}\in\Y^\Z$. Clearly, for any $\mathbf{y}$ fixed,
	$Z_{-j,0}^{x_0,\mathbf{y}}$ is $\G_0$ measurable and by inequality
	\eqref{eq:extra}, Lemma \ref{lem:central} is applicable for proving that for $\law{Y}$-$\P$-a.s.{}
	$\mathbf{y}\in\Y^\Z$, the law of $Z_{-j,pt}^{x_0,\mathbf{y}}$
	converges to a probability law as $t\to\infty$, and the limit
	is independent of the choice of $j$. From this follows that
	Theorem \ref{thm:TV} applies for $\left(\tilde{X}_t^j\right)_{t\in\N}$, $j\in\{0,\ldots,p-1\}$ moreover, 
	the limiting laws, which are subsequences of $\left(\mu_n=\law{X_n}\right)_{n\in\N}$, are equal. But then
	the whole sequence
	$\left(\mu_n\right)_{n\in\N}$ converges to some $\mu_{\ast}$ in total variation.
	
	The environment $\tilde{Y}_t$, $t\in\Z$ is also ergodic hence 
	by Theorem \ref{thm:LLN}, we have a law of large numbers in $L^q$ for 
	each $\left(\Phi(\tilde{X}_t^j)\right)_{t\in\N}$, $j\in\{0,\ldots,p-1\}$
	and $1\le q <\infty$.
	So, by the triangle inequality, we can write
	\begin{align*}
	\left\Vert
	\frac{1}{N}\sum_{k=1}^{N}\Phi (X_k)
	-
	\int\limits_{\X}
	\Phi (z)
	\mu_{\ast}(\dint z)
	\right\Vert_q
	\le
	\frac{1}{p}
	\sum_{j=0}^{p-1}
	\left\Vert
	\frac{p}{N}
	\sum_{k=0}^{\lfrf{\frac{N}{p}}}
	\Phi (\tilde{X}_k^j)
	-
	\int\limits_{\X}
	\Phi (z)
	\mu_{\ast}(\dint z)
	\right\Vert_q
	+
	O(N^{-1})
	\end{align*}
	showing that the law of large numbers holds true in $L^q$ also
	for $\Phi (X_t)$. This completes the proof.\hfill $\Box$

	\section{Appendix}\label{lm}

For the reader's convenience, we recall a concept of mixing defined in \cite{Lmixing} 
which was used in some of the estimations above.
Let $\mathcal{G}_t$, $t\in\N$ be an increasing sequence of sigma-algebras and let $\mathcal{G}^+_t$, $t\in\N$
be a decreasing sequence of sigma-algebras such that, for each $t\in\N$, 
$\mathcal{G}_t$ is independent of $\mathcal{G}^+_t$.

Let  $W_t$, $t\in\N$ be a real-valued stochastic process. For each $p\geq 1$, introduce
$$
M_p(W):=\sup_{t\in\N} E^{1/p}[|W_t|^p].
$$
For each process $W$ such that $M_1(W)<\infty$ define, for each $p\geq 1$, 
$$
\ga_p(W,\tau):=\sup_{t\geq\tau}E^{1/p}[|W_t-E[W_t|\mathcal{G}_{t-\tau}^+]|^p],\ \tau\in\N,\ 
\Gamma_p(W):=\sum_{\tau=0}^{\infty} \ga_p(W,\tau).
$$

For some $p\geq 1$, the process $W$ is called 
\emph{$L$-mixing of order $p$} with respect to $(\mathcal{G}_t,\mathcal{G}^+_t)$, $t\in\N$ if
it is adapted to $(\mathcal{G}_t)_{t\in\N}$ and $M_p(W)<\infty$, $\Gamma_p(W)<\infty$. We say that $W$
is \emph{$L$-mixing} if it is $L$-mixing of order $p$ for all $p\geq 1$.

We recall Lemma 2.1 of \cite{Lmixing}.

\begin{lemma}\label{fyffes} Let 
$\mathcal{G}\subset\mathcal{F}$ be a sigma-algebra,
$X$, $Y$ random variables with $E^{1/p}[|X|^p]<\infty$, $E^{1/p}[|Y|^p]<\infty$ with some $p\geq 1$.
If ${Y}$ is $\mathcal{G}$-measurable then
$$
E^{1/p}[|X-E[X\vert\mathcal{G}]|^p]\leq 2E^{1/p}[|X-Y|^p]
$$
holds.\hfill $\Box$ 
\end{lemma}

Finally, a trivial consequence of Theorem 1.1 of \cite{Lmixing} is formulated.

\begin{lemma}\label{inek}
For an $L$-mixing process $W$ of order $p\geq 2$ satisfying $E[W_t]=0$, $t\in\N$,
$$
E^{1/p}\left[\left|\sum_{i=1}^N W_i\right|^p\right]\leq C_p N^{1/2} M_p^{1/2}(W)\Gamma_p^{1/2}(W),
$$
holds for each $N\geq 1$ with a constant $C_p$ that does not depend either on $N$ or on $W$.
\hfill $\Box$
\end{lemma}

	%\bibliography{tedeum1}
	\bibliographystyle{plain}

\end{document}